\documentclass[12pt]{amsart}

\usepackage{fullpage, graphicx}
\usepackage{mathrsfs}
\usepackage{shuffle}
\usepackage{tikz}
\usepackage{array}
\usepackage{appendix}
\usepackage{cite}
\usepackage{subfigure}
\usepackage{caption}
\usepackage{subcaption}
\usepackage[hidelinks]{hyperref}
\usepackage{placeins}
\usepackage{amssymb}

\usetikzlibrary{decorations.pathreplacing}
\newtheorem{thm}{Theorem}
\newtheorem{prop}[thm]{Proposition}

\newtheorem{lemma}[thm]{Lemma}
\newtheorem{remark}[thm]{Remark}
\newtheorem{cor}[thm]{Corollary}
\theoremstyle{definition}
\newtheorem{definition}[thm]{Definition}

\newcommand{\lc}[1]{\tilde{#1}}

\newcommand{\precdot}{\prec\mathrel{\mkern-5mu}\mathrel{\cdot}}
\newcommand{\succdot}{\ \cdot\mathrel{\mkern-10mu}\mathrel{\succ}} 

\title{Hopf algebras from poset growth models}

\author{Karen Yeats${}^{a}$}
\address{${}^{a}$Department of Combinatorics and Optimization, University of Waterloo,
  Waterloo, ON, Canada}
\email{kayeats@uwaterloo.ca}

\author{Stav Zalel${}^{b}$}
\address{${}^{b}$Blackett Laboratory, Imperial College London, SW7 2AZ, U.K}
\email{stav.zalel11@imperial.ac.uk}

\thanks{Thanks to the two referees for their useful comments and particularly for the correction to Corollary~\ref{cor nice coprod}.\linebreak
KY is supported by an NSERC Discovery grant and the Canada Research Chairs program and is grateful for the hospitality of Perimeter Institute. Research at Perimeter Institute is supported in part by the Government of Canada through the Department of Innovation, Science and Economic Development Canada and by the province of Ontario through the Ministry of Economic Development, Job Creation and Trade. This research was also supported in part by the Simons Foundation through the Simons Foundation Emmy Noether Fellows Program at Perimeter Institute. SZ is supported by STFC grant ST/W006537/1 and STFC Consolidated Grant ST/X000575/1.}

\begin{document}

\begin{abstract}
We give a framework for growth models on posets which simultaneously generalizes the Classical Sequential Growth models for posets from causal set theory and the tree growth models of natural growth and simple tree classes, the latter of which also appear as solutions of combinatorial Dyson-Schwinger equations in quantum field theory.  We prove which cases of the Classical Sequential Growth models give subHopf algebras of the Hopf algebra of posets, in analogy to a characterization due to Foissy in the Dyson-Schwinger case. We find a family of generating sets for the Connes-Moscovici Hopf algebra.
\end{abstract}

\maketitle
\tableofcontents
\section{Introduction}

In the causal set approach to quantum gravity, spacetime has two fundamental attributes: a causal structure and a fundamental discreteness. Thus, this approach proposes that spacetime is a locally finite poset (also known as a ``causal set''), where the partial order is interpreted as the causal order of spacetime \cite{Bombelli:1987aa,Surya:2019ndm,Dowker:2006sb}. 

An important challenge shared by all approaches to quantum gravity is that of defining a dynamics for spacetime. In this respect, the discreteness of the causal set approach is both a blessing and a curse. On the one hand, the lack of a continuum time parameter renders the canonical approach and the continuum Lagrangian formalism obsolete. On the other hand, some of the technical hurdles which one encounters in the path integral approach in the continuum are tamed in the discrete. Taking its cue from this observation, one line of research investigates models of random causal sets as models of dynamical spacetime \cite{Dowker:2022ehl,Dowker:2010qh,Ash:2002un,Georgiou:2010,Brightwell:2002vw,Martin:2000js,Rideout:1999ub,Surya:2020cfm,Dowker:2019qiz}. In particular, one is interested in those random models which are \textit{physically motivated}. Broadly, physical motivations can include the recovery of General Relativity in an appropriate large-scale approximation, as well as causality and covariance. However, the exact formulation of such conditions is an open problem, especially in the case of bona fide quantum dynamics. In \cite{Rideout:1999ub}, a concrete formulation of physical conditions suitable for a discrete setting was proposed and solved to give a family of random causal set models, known as the Classical Sequential Growth (CSG) models, see Section~\ref{sec cst}. Since their proposal, the CSG models have become the archetype for causal set spacetime dynamics and their study has led to interesting insights, for example in quantum cosmology \cite{Sorkin:1998hi,Dowker:2017zqj}, observables in quantum gravity \cite{Dowker:2022ehl,Brightwell:2002vw}, quantum spacetime dynamics \cite{Dowker:2010qh,Surya:2020cfm} and combinatorics \cite{Georgiou:2005,Georgiou:2010, Brightwell:2016, Brightwell:2011, Brightwell:2012}. 

Here, we study the relationship between CSG models and Hopf algebras of finite posets. Our work is motivated by the existing connections between physics and Hopf algebras -- the Connes-Kreimer Hopf algebra of rooted trees is known to give a rigorous underpinning to the process of renormalization in quantum field theory \cite{Khopf, ckI, ckII} -- as well as by their shared combinatorial language.

For our context, combinatorial Hopf algebras can be thought of as follows.  If we have two discrete objects, then a product for them will tell us how to combine the two objects together into a new object, or possibly into a sum of new objects obtained from different ways of combining the two original objects together.  Disjoint union is an example of a possible multiplication that will be especially useful for us. Dually, a coproduct will tell us how to take one discrete object and break it into two, potentially in a single way, or potentially as a sum of multiple ways.  With suitable compatibilities, the product and coproduct together give a Hopf algebra structure on these objects, see Section~\ref{sec hopf}. 
We will work with the Hopf algebra of finite posets in which the product is given by disjoint union and the coproduct is given by decompositions into a down-set and the complimentary up-set, see definition \ref{def coproduct on posets} in Section~\ref{sec hopf}.  This is a well-known Hopf algebra in the combinatorics community, see section 13.1 of \cite{AMmonoidal}.  

Rooted trees and forests can be seen as special cases of finite posets in two ways, with roots taken to be either minimal or maximal elements of the poset (Fig.\ref{fig growth}).
Depending on which perspective is taken, then, one ``grows'' rooted trees either by adding leaves to a tree or adding a new root to a forest. Both of these types of growth are important in the combinatorial study of rooted trees. Growing by the addition of leaves appears in particular in the realization of the Connes-Moscovici Hopf algebra by rooted trees \cite{ck0} (see \cite{cmhopf} for the original formulation), while growing by the addition of roots is even more fundamental to what rooted trees are since it is the basis for the standard recursive definition of a rooted tree and for much of tree enumeration, see for instance \cite{FSbook} section I.5 for an introduction.  Following \cite{MMsimple} we call classes of trees which are grown by roots in a precise way that is determined by composition with a formal power series (see \eqref{eq combDSE simple tree}) simple tree classes.  Interpreted in the quantum field theory context, these are classes of trees which come from Dyson-Schwinger equations, see \cite{Ybook} and references therein. In \cite{FoissyDyson} Foissy characterized precisely which simple tree classes give a subHopf algebra of the Connes-Kreimer Hopf algebra.

Inspired by this context, in this work we present a framework for generalized growth models of posets of which the CSG models, the Connes-Moscovici Hopf algebra of rooted trees, and all simple tree classes are special cases, see Section~\ref{sec set up}. Our framework takes the form of a recursive definition (see \eqref{eq general set up}), akin to the form of the simple tree classes, but with a new operator implementing the growth (Definition~\ref{def M}). As well as a base case, our recursive formula intakes a countable set of a parameters.

We prove exactly when our framework gives subHopf algebras in two regions of parameter space, corresponding to the simple tree classes and the CSG models (Theorem \ref{thm for general set up}). We find that the Transitive Percolation models (a sub-family of the CSG models) give rise to co-commutative Hopf algebras, and we give a closed-form expression for their coproduct coefficients (Lemma \ref{tp lemma}). We prove that the so-called ``Forest'' CSG models give rise to Hopf algebras isomorphic to the Connes-Moscovici Hopf algebra and where the isomorphisms themselves are interesting since they pick out different sets of generators (Lemma \ref{lem forest hopf by iso}). Thus we find a new family of combinatorially-meaningful generating sets for the Connes-Moscovici Hopf algebra and we give a recursive expression for their co-product coefficients. As a special case of our result, we find a new expression for the coproduct coefficients of the usual Connes-Moscovici generators. We conclude with some comments on the application of our results within the causal set approach to quantum gravity as well as some possible future directions.

\begin{figure}
\subfigure[]{
\begin{tikzpicture}
\filldraw[black] (0,0) circle (2pt) node[anchor=west]{};
 \draw [-stealth](0.2,0) -- (.8,0);
 \filldraw[black] (1,0) circle (2pt) node[anchor=west]{};
 \draw [-stealth](1.2,0) -- (1.8,0);
  \draw [black, thick](1,0) -- (1,0.75);
   \filldraw[black] (1,0.75) circle (2pt) node[anchor=west]{};
 \filldraw[black] (2,0) circle (2pt) node[anchor=west]{};
 \draw [-stealth](2.2,0) -- (3,0);
  \draw [black, thick](2.25,0.75) -- (2,0);
   \filldraw[black] (2.25,0.75) circle (2pt) node[anchor=west]{};
  \draw [black, thick](1.75,0.75) -- (2,0);
   \filldraw[black] (1.75,0.75) circle (2pt) node[anchor=west]{};
   \filldraw [black, thick](3.75,0)circle (2pt) node[anchor=west]{};
    \filldraw[black] (3,0.75) circle (2pt) node[anchor=west]{};
   \draw [black, thick](3.,0.75) -- (3.25,0);
    \filldraw[black] (3.5,0.75) circle (2pt) node[anchor=west]{};
      \draw [black, thick](3.5,0.75) -- (3.25,0);
  \filldraw[black] (3.25,0) circle (2pt) node[anchor=west]{};
\end{tikzpicture}
\label{fig growth leaves}
}
\hspace{4cm}
\subfigure[]{
\begin{tikzpicture}
\filldraw[black] (0,0) circle (2pt) node[anchor=west]{};
 \draw [-stealth](0.2,0) -- (.8,0);
 \filldraw[black] (1,0) circle (2pt) node[anchor=west]{};
 \draw [-stealth](1.2,0) -- (1.8,0);
  \draw [black, thick](1,0) -- (1,0.75);
   \filldraw[black] (1,0.75) circle (2pt) node[anchor=west]{};
   \filldraw[black] (2,0) circle (2pt) node[anchor=west]{};
    \filldraw[black] (2,0.75) circle (2pt) node[anchor=west]{};
     \draw [black, thick](2,0) -- (2,0.75);
      \filldraw[black] (2.5,0) circle (2pt) node[anchor=west]{};
       \draw [-stealth](2.7,0) -- (3.3,0);
  \filldraw[black] (3.5,0) circle (2pt) node[anchor=west]{};
    \filldraw[black] (3.5,0.75) circle (2pt) node[anchor=west]{};
     \draw [black, thick](3.5,0) -- (3.5,0.75);
      \filldraw[black] (4,0) circle (2pt) node[anchor=west]{};
      \filldraw[black] (3.5,1.5) circle (2pt) node[anchor=west]{};
     \draw [black, thick](3.5,1.5) -- (4,0);
     \draw [black, thick](3.5,1.5) -- (3.5,0.75);
\end{tikzpicture}
\label{fig growth roots}
}
\caption{Growing forests. In \ref{fig growth leaves}, a forest is grown by adding leaves, so the roots are the minimal elements of the poset. In \ref{fig growth roots}, a forest is grown by adding roots, so the roots are the maximal elements of the poset.}\label{fig growth}

\end{figure}

\section{Background}

\subsection{Posets}\label{subsecposetbackground}

A \emph{partially ordered set} or \emph{poset} is a set $P$ with a reflexive, antisymmetric, and transitive relation, usually written $\preceq$, on it. We will use $\preceq$ to denote various partial orders, the meaning should be clear from the context. We reserve the symbol $\leq$ for the total order on the integers. By the standard abuse of notation we also write $P$ for the poset, that is, we will use the same notation for a poset and its underlying set.

Given two elements $x,y$ from a poset $P$, the \emph{interval} defined by $x$ and $y$, written $[x,y]$ is
\[
[x,y]=\{z:x\preceq z\preceq y\}
\]
which inherits a poset structure from $P$.

A poset is \emph{finite} if the underlying set is finite.  A poset is \emph{locally finite} if every interval of the poset is finite. A \textit{causal set} or \textit{causet} is a locally finite poset. We will be concerned primarily with finite posets and all posets will be assumed to be finite unless otherwise specified. Thus, the terms poset and causet are interchangable in our setting.

An element $x$ of a poset $P$ is a \emph{maximal} element if there is no element $y\neq x$ with $x \preceq y$.  Likewise an element $x$ of a poset $P$ is a \emph{minimal} element if there is no element $y\neq x$ with $y\preceq x$.

Given a poset $P$ and $x,y \in P$, we say that there is a \emph{relation} between $x$ and $y$ if $x\prec y$ or $y\prec x$. We say that $x$ \emph{covers} $y$ and write $x\succdot y$ if $x\succeq y$, $x\neq y$ and there is no element $z\neq x,y$ with $x\succeq z\succeq y$. To put it more informally, $x$ covers $y$ if $x$ is larger than $y$ but there is nothing between them.  When $x$ covers $y$ we also say there is a \emph{link} between $x$ and $y$.

A poset is often visualized via its \emph{Hasse diagram}. The Hasse diagram of a poset $P$ is a drawing of the graph whose vertices are the elements of $P$ and whose edges are given by the cover relation, where if $x \succdot y$ then $x$ is drawn above $y$.

A \emph{rooted tree} is a connected acyclic graph with one vertex marked as the \emph{root}, or equivalently, a rooted tree is a vertex $r$ called the root with a multiset of rooted trees whose roots are the \emph{children} of $r$.  A \emph{forest} (of rooted trees) is a disjoint union of rooted trees.  Rooted trees and forests can be seen as special cases of posets in two ways. The first way is as a poset where the root is a maximal element and every non-root element has exactly one element covering it.  Then the cover relation gives the parent-child relation, where if $x\succdot y$ then $x$ is the parent of $y$, and the Hasse diagram is the rooted tree as a graph, with the root at the top. The other way is as a poset where the root is a minimal element and every non-root element covers exactly one other element. In this case the Hasse diagram is again the rooted tree as a graph but with the root at the bottom and the parent-child relation moving upwards.


\begin{definition}[Down-set, up-set, component]\label{def down set}
  Let $P$ be a poset.
  \begin{itemize}
  \item A \emph{down-set} of $P$ is a set $D$ of elements of $P$ with the property that if $x\in D$ and $y\in P$ with $y\preceq x$ then $y\in D$.  $D$ inherits a poset structure from $P$ and so by the usual abuse of notation we will also write $D$ for the down-set as a poset.  Down-sets are known as \emph{stems} in causal set theory and are sometimes also called \emph{lower sets} or \emph{ideals} in mathematics. 
  \item An \emph{up-set} of $P$ is a set $U$ of elements of $P$ with the property that if $x\in U$ and $y\in P$ with $y\succeq x$ then $y\in U$.  $U$ inherits a poset structure from $P$ and again we will also write $U$ for this poset.  Up-sets are also sometimes called \emph{upper sets} or \emph{filters}.
  \item Given a subset $S$ of the elements of $P$.  Write
    \begin{align*}
      D(S) & = \{y\in P: \exists x\in S, y\preceq x\} \\
      U(S) & = \{y\in P: \exists x\in S, y\succeq x\}
    \end{align*}
    for the down-set and up-set (respectively) generated by $S$.
    In causal set theory $D(S)$ is the \emph{inclusive past} of $S$ and $U(S)$ is the \emph{inclusive future} of $S$. 
    \item $S\subseteq P$ is a \emph{component} of ${P}$ if $S$ is both a down-set and an up-set of $P$ and there is no nonempty $S'\subsetneq S$ which is both a down-set and an up-set of $P$. $P$ is the disjoint union of its components and the decomposition of $P$ into its components is unique. $P$ is \emph{connected} if it consists of exactly one component. In a Hasse diagram, each component of $P$ appears as a connected component in the sense of graph theory and so a connected poset is a poset whose Hasse diagram is connected as a graph. If a poset $P$ is a forest then its components are the trees it contains.
  \end{itemize}
\end{definition}

\begin{definition}[Isomorphism, labellings, unlabelled poset, template]\label{def template}
\

  \begin{itemize}
 \item Two posets $P_1$ and $P_2$ are \emph{isomorphic} if there is a bijection between their underlying sets that preserves the order relation.  That is, if there is a bijection $f:P_1\rightarrow P_2$ such that $f(x)\preceq f(y) \Leftrightarrow x\preceq y$.
  \item 
    We will call a poset $P$ \emph{increasingly labelled} if the underlying set is $P\subseteq\mathbb{Z}_{\geq 1}$ and whenever $x\preceq y$ in $P$ then $x\leq y$. If $P$ is increasingly labelled, $x\in P$ is the $n^{th}$ element of $P$ if there exist exactly $n$ distinct elements $y\in P$ such that $y\leq x$. We will call a poset $P$ \emph{naturally labelled} if it is increasingly labelled and has underlying set $P=[1,|P|]$. So, if $P$ is naturally labelled, the $n^{th}$ element of $P$ is $n$.  \item We will call a poset \emph{unlabelled} when we consider it only up to isomorphism, or more formally, the unlabelled posets are the equivalence classes of posets under isomorphism.  An unlabelled poset is generally drawn by giving its Hasse diagram without any labels on the vertices. We will denote the set of finite unlabelled posets by $\mathcal{P}$.
  \item The \emph{cardinality of an unlabelled poset $P$} is the cardinality of any of the labelled representatives of $P$. The \emph{components of an unlabelled poset $P$} are the unlabelled posets represented by the connected components of the Hasse diagram of $P$, or equivalently, are the equivalence classes of the components of any labelled representative of $P$.
  \item  Given an unlabelled poset $P$, the \emph{templates} of $P$ are its naturally labelled representatives. We call them templates because, given a template $T$ and an underlying set $\{i_1,\ldots,i_{|P|}\}\subseteq\mathbb{Z}_{\geq 1}$ with $i_1<i_2<\cdots<i_{|P|}$, one can construct an increasingly labelled representative $L$ of $P$ by arranging the elements of the underlying set according to the template by setting $i_x\preceq i_y$ in $L$ whenever $x\preceq y$ in $T$. We denote the set of templates of $P$ by $temp(P)$. An example is shown in Fig.\ref{fig: template example}.
  \item  The \emph{number of natural labelings} or \emph{number of templates} of an unlabelled poset $P$ is $\Psi(P)=|temp(P)|$.
\end{itemize}
\end{definition}

\begin{figure}[ht]
\centering
     \subfigure
         \centering
        \includegraphics[scale=.5]{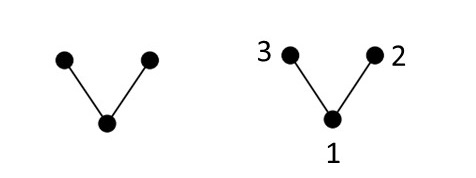}
\label{templates_example1}
     \hfill
 \subfigure
         \centering
        \includegraphics[scale=0.45]{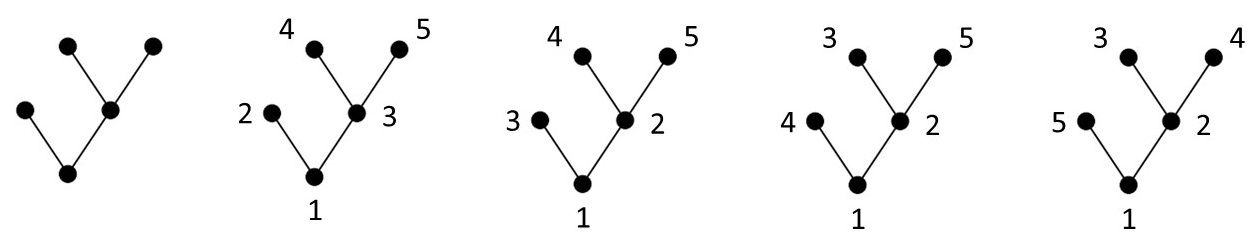}
         \label{templates_example2}
\caption{An illustration of the interplay between unlablled posets and their templates (labelled representatives). Two unlabelled posets are shown with their templates. The ``V'' poset has a unique template, while the 5-element poset shown has four templates.}\label{fig: template example}
\end{figure}

%

The notion of natural labelling is also known as heap ordering or linear extension in other contexts.

The interplay between unlabelled and increasingly labelled posets will be quite important in what follows. However, for our context the unlabelled posets are the default and we will use the term \textit{poset} to mean unlabelled poset, unless specified otherwise. To remove any ambiguity, we will take the notational convention of using capital roman letters, e.g. $P, C$, for unlabelled posets and capital roman letters with a tilde \ $\lc{}$ \ , e.g. $\lc{P}, \lc{C}$, for increasingly labelled posets. Unless otherwise specified, for a poset or increasingly labelled poset $C_k, \lc{C}_k$ the lower index will indicate the cardinality of the poset. Any other indexing will be indicated with an upper index.

\begin{definition}[Forest partitions]\label{def forest partition}
    A \emph{partition} of a forest $F$ is a multiset of forests whose union is $F$. The coarsest partition of $F$ contains only $F$, while the finest partition contains the components of $F$. Given a forest $F$ with some partition $\pi$, we write $N(\pi)$ to denote the multiset of integers whose entries are the cardinalities of the forests in $\pi$. 
\end{definition}

An illustration is shown in table \ref{table partition example}.

It will be useful later to count the number of times an integer appears in $N(\pi)$ and the number of times a given forest appears in $\pi$. When $\pi$ is the finest partition of $F$, the latter is equivalent to counting the number of times a tree appears as a component in $F$. For consistency, we combine these various notions of counting into the notion of multiplicity.
\begin{definition}[Multiplicity]\label{def multiplicity} We write $\mu^y(x)$ to denote the multiplicity of $x$ in $y$ in the following contexts:
\begin{itemize}
\item When $y$ is a multiset, $\mu^y(x)$ is the number of times $x$ appears as an element in $y$. Specifically, we will consider multisets whose elements are posets or whose elements are positive integers.
\item When $y$ is a poset and $x$ is a connected poset, $\mu^y(x)$ is the number of times $x$ appears as a component in $y$.
\end{itemize}
To simplify our notation, we may write $\mu(x)$ when $y$ can be understood from the context.
\end{definition}

\begin{table}[ht]
\begin{tabular}{ |m{3cm}|m{3cm}|m{3cm}|}   
   \hline
   $\pi$&$N(\pi)$&$\mu^{N(\pi)}(2)$\\
   \hline
   \begin{tikzpicture}
\node[align=left] at (-2.2,-0.4) {$\{$};
\filldraw[black] (-2,-0.25) circle (1.5pt) node[anchor=west]{};
\filldraw[black] (-2,-0.5) circle (1.5pt) node[anchor=west]{};
\draw[black, thick] (-2,-0.25) -- (-2,-0.5);
\filldraw[black] (-1.8,-0.5) circle (1.5pt) node[anchor=west]{};
\filldraw[black] (-1.6,-0.5) circle (1.5pt) node[anchor=west]{};
\node[align=left] at (-1.4,-0.4) {$\}$};
\end{tikzpicture}&4&0\\
  \hline
 \begin{tikzpicture}
\node[align=left] at (-2.2,-0.4) {$\{$};
\filldraw[black] (-2,-0.25) circle (1.5pt) node[anchor=west]{};
\filldraw[black] (-2,-0.5) circle (1.5pt) node[anchor=west]{};
\draw[black, thick] (-2,-0.25) -- (-2,-0.5);
\filldraw[black] (-1.8,-0.5) circle (1.5pt) node[anchor=west]{};
\filldraw[black] (-1.4,-0.5) circle (1.5pt) node[anchor=west]{};
\node[align=left] at (-1.2,-0.4) {$\}$};
\node[align=left] at (-1.6,-0.6) {$,$};
\end{tikzpicture}& 3,1&0\\
  \hline
 \begin{tikzpicture}
\node[align=left] at (-2.2,-0.4) {$\{$};
\filldraw[black] (-2,-0.25) circle (1.5pt) node[anchor=west]{};
\filldraw[black] (-2,-0.5) circle (1.5pt) node[anchor=west]{};
\draw[black, thick] (-2,-0.25) -- (-2,-0.5);
\filldraw[black] (-1.6,-0.5) circle (1.5pt) node[anchor=west]{};
\filldraw[black] (-1.4,-0.5) circle (1.5pt) node[anchor=west]{};
\node[align=left] at (-1.2,-0.4) {$\}$};
\node[align=left] at (-1.8,-0.6) {$,$};
\end{tikzpicture}&2,2&2\\
  \hline
 \begin{tikzpicture}
\node[align=left] at (-2.2,-0.4) {$\{$};
\filldraw[black] (-2,-0.25) circle (1.5pt) node[anchor=west]{};
\filldraw[black] (-2,-0.5) circle (1.5pt) node[anchor=west]{};
\draw[black, thick] (-2,-0.25) -- (-2,-0.5);
\filldraw[black] (-1.6,-0.5) circle (1.5pt) node[anchor=west]{};
\filldraw[black] (-1.2,-0.5) circle (1.5pt) node[anchor=west]{};
\node[align=left] at (-1.,-0.4) {$\}$};
\node[align=left] at (-1.8,-0.6) {$,$};
\node[align=left] at (-1.4,-0.6) {$,$};
\end{tikzpicture}&2,1,1&1\\
  \hline
  \end{tabular}
  \caption{Illustration of definitions \ref{def forest partition} and \ref{def multiplicity}. The partitions $\pi$ of the forest 
 \begin{tikzpicture}
\protect\filldraw[black] (-2,-0.25) circle (1.5pt) node[anchor=west]{};
\protect\filldraw[black] (-2,-0.5) circle (1.5pt) node[anchor=west]{};
\protect\draw[black, thick] (-2,-0.25) -- (-2,-0.5);
\protect\filldraw[black] (-1.8,-0.5) circle (1.5pt) node[anchor=west]{};
\protect\filldraw[black] (-1.6,-0.5) circle (1.5pt) node[anchor=west]{};
\end{tikzpicture}, the associated lists of integers $N(\pi)$, and the multiplicity $\mu^{N(\pi)}(2)$ of $2$ in $N(\pi)$ are shown.}\label{table partition example}
\end{table}

The following relationships between the number of templates of a poset and its components will be useful to us.

\begin{lemma} Let $C_k$ be a poset with $d$ components, $C_{k_1}^1,\cdots,C_{k_d}^d$ (so that $k_1+\cdots+k_d=k$). Then the relationship between the number of templates of $C_k$, $\Psi(C_k)$, and the number of templates of its components, $\Psi(C^i_{k_i})$, is given by,
\begin{equation}\label{component_template_number}\Psi(C_k)=\binom{k}{k_1,k_2\cdots k_d}\frac{\prod_{i=1}^d \Psi(C_{k_i}^i)}{\prod_{P\in\mathcal{C}}\mu(P)!}=\frac{k!\prod_{i=1}^d \Psi(C_{k_i}^i)}{\prod_{i=1}^d k_i! \prod_{P\in\mathcal{C}}\mu(P)!},\end{equation}
where $\mathcal{C}$ is the set of connected finite posets and $\mu(P)=\mu^{C_k}(P)$.

When $C_k$ is a forest, the relationship between the number of templates of $C_k$ and the number of templates of the forests in any partition $\pi$ of $C_k$ is given by, 
\begin{equation}\begin{split}\label{nooftemplatesforest}\Psi(C_k)&=k! \prod_{F\in\pi}\frac{\Psi(F)}{|F|!}\prod_{P\in\mathcal{C}}\frac{\prod_{F\in\pi}\mu^F(P)!}{\mu(P)!}\\
&=k! \prod_{F\in\mathcal{P}}\bigg(\frac{\Psi(F)}{|F|!}\bigg)^{\mu^\pi(F)}\prod_{P\in\mathcal{C}}\frac{\prod_{F\in\mathcal{P}}\big(\mu^F(P)!\big)^{\mu^\pi(F)}}{\mu(P)!},\end{split}\end{equation}where $\mathcal{P}$ is the set of finite posets. One can verify that equation \eqref{nooftemplatesforest} reduces to equation \eqref{component_template_number} when $\pi$ is the unique partition of $C_k$ into its component trees.
\end{lemma}

\begin{proof}
    To arrive at \eqref{component_template_number}, note that the multinomial coefficient counts the number of ways of assigning a subset of the interval $[1,k]$ as a set of labels to each component; the product of $\Psi$'s counts the number of ways of arranging these labels within each component; and the denominator ensures that no over-counting takes place when a connected poset $P$ appears as a component in $C_k$ more than once.

        To see that \eqref{nooftemplatesforest} holds, start with its RHS and use \eqref{component_template_number} to write $\Psi(F)$ in terms of $\Psi(t)$, where $t$ are the trees in $F$. After simplifying, one finds that the resulting expression is equal to $\Psi(C_k)$ by \eqref{component_template_number}.
\end{proof}

\subsection{Classical Sequential Growth}\label{sec cst}

The Classical Sequential Growth (CSG) models are the archetypal toy-models of spacetime dynamics within causal set theory. The CSG models are models of random posets in which a poset grows stochastically through an accretion of elements. The growth happens in stages. At each stage, an element is born into the poset, forming relations with the existing elements subject to the rule that the new element cannot be made to precede any of the existing elements in the partial order.

Keeping track of which element was born at which stage is tantamount to the statement that the CSG models grow naturally labelled posets. Indeed, we can think of stage $n$ of the process as a transition $\lc{C}_n\rightarrow\lc{C}_{n+1}$ from a poset with $n$ elements to a poset with $n+1$ elements, where both posets are naturally labelled and $\lc{C}_n$ is a down-set in $\lc{C}_{n+1}$. The down-set relation induces a partial order on the set of finite naturally labeled posets, \textit{i.e.}  $\lc{C}_n\preceq\lc{C}_{m} \iff$ $\lc{C}_n$ is a down-set in $\lc{C}_{m}$. This partial order of finite posets (known in the causal set literature as \emph{poscau}) is a tree, and each stage of the growth is a transition from a parent $\lc{C}_n$ to one of its children $\lc{C}_{n+1}$. A CSG model is a set of transition probabilities $\mathbb{P}(\lc{C}_n\rightarrow\lc{C}_{n+1})$, one probability for each parent-child pair.

A CSG model is specified by a countable sequence $(t_0,t_1,t_2,\ldots)$ of real non-negative numbers (or ``couplings'') from which the transition probabilities are obtained via,
\begin{equation}\begin{split}\label{transprob1}
&\mathbb{P}(\tilde{C}_n \rightarrow \tilde{C}_{n+1})= \frac{\lambda(\varpi, m)}{\lambda(n, 0)} \ \forall n\geq 1\\
\end{split}\end{equation}
\begin{equation}\begin{split}\label{transprob2}
\lambda(k, p) = \sum_{i = 0}^{k - p} \binom{k-p}{i} t_{p+i},
\end{split}\end{equation}
where $\varpi$ and $m$ are the number of relations and links, respectively, formed by the new element. An illustration is given in Fig.\ref{transition_prob_example} and a more intuitive understanding of the form of these transition probabilities is discussed later in Example~\ref{rem CSG by proto-past}.
\begin{figure}[htpb]
  \centering
	\includegraphics[width=0.4\textwidth]{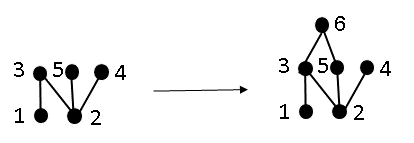}
	\caption[CSG transition probabilities]{An illustration of the CSG transition probabilities. A transition between a parent and its child is shown. During the transition, the new element $6$ forms $m=2$ links (with elements 3 and 5) and $\varpi=4$ relations (with elements 1,2,3 and 5). Thus the probability for the transition is equal to $\frac{\lambda(4, 2)}{\lambda(5, 0)}$, as given by equation \eqref{transprob1}.}
	\label{transition_prob_example}
\end{figure}
The transition probabilities are normalised so that they satisfy the Markov sum rule,
\begin{equation}\label{eq markov}\sum_{\tilde{C}_{n+1}}\mathbb{P}(\tilde{C}_n \rightarrow \tilde{C}_{n+1})=1,\end{equation}
where the sum is over all children $\tilde{C}_{n+1}$ of a fixed parent $\tilde{C}_{n}$. The probability $\mathbb{P}(\lc{C}_n)$ of growing some poset $\lc{C}_n$ is given by,
\begin{equation}\label{eq_csg_prob_1}\begin{split}
&\mathbb{P}(\lc{C}_1)=1,\\
&\mathbb{P}(\tilde{C}_n)=\prod_{i=1}^{n-1} \mathbb{P}(\tilde{C}_i \rightarrow \tilde{C}_{i+1}) \ \forall n\geq 1,\\
\end{split}
\end{equation}
where for each $i$, $\lc{C}_i$ is the unique parent of $\lc{C}_{i+1}$. The CSG models possess the property of \emph{discrete general covariance}\footnote{The CSG models are the unique solution to the simultaneous conditions of discrete general covariance and ``bell causality'', a condition relating ratios of transition probabilities \cite{Rideout:1999ub}.}, namely that given a pair of isomorphic posets $\lc{C}_n$ and $\lc{C}_n'$, we have $\mathbb{P}(\tilde{C}_n)=\mathbb{P}(\tilde{C}_n')$. Given an unlabelled poset $C_n$, the probability $\mathbb{P}(C_n)$ assigned to $C_n$ by a CSG model is,
\begin{equation}\label{eqtn0606231}\mathbb{P}(C_n) = \sum_{\tilde{C}_n\in temp(C_n)}\mathbb{P}(\tilde{C}_n)=\Psi(C_n)\mathbb{P}(\lc{C}_n),\end{equation}
where on the right hand side $\lc{C}_n$ is any template of $C_n$ (cf. definition \ref{def template}). The second equality follows from discrete general covariance. The probabilities satisfy the sum rule,
\begin{equation}\label{sum rule 2}\begin{split}
&\sum_{\lc{C}_n}\mathbb{P}(\lc{C}_n)=\sum_{C_n}\mathbb{P}({C}_n)=1,\\
\end{split}
\end{equation} where the sums are 
over all labelled posets of cardinality $n$ and unlabelled posets of cardinality $n$, respectively.  

Sometimes we will be interested in  \emph{relative} probabilities rather than in normalised probabilities. For this purpose, we define the weight $w$ as,
\begin{equation}\begin{split}\label{weight}
&w(\tilde{C}_n \rightarrow \tilde{C}_{n+1})= \lambda(\varpi, m) \ \forall n\geq 1,\\
&w(\lc{C}_1)=1,\\
&w(\tilde{C}_n)=\prod_{i=1}^{n-1} w(\tilde{C}_i \rightarrow \tilde{C}_{i+1}),\ \forall n\geq 1,\\
&w(C_n)=\sum_{\tilde{C}_n\in temp(C_n)}w(\tilde{C}_n)=\Psi(C_n)w(\lc{C}_n),
\end{split}\end{equation}
where $\varpi$ and $m$ are as in \eqref{transprob1}. We can define the weight of a transition between two unlabelled posets $w(C_n\rightarrow C_{n+1})$ as the sum of transition weights $w(\lc{C}_{n}\rightarrow\lc{C}_{n+1})$ over representatives $\lc{C}_{n+1}$ of ${C}_{n+1}$ given a fixed representative $\lc{C}_n$ of $C_n$. Writing $\Psi(C_{n+1}|C_n)$ to denote the number of ways of extending a fixed natural labelling of $C_n$ to a natural labelling of $C_{n+1}$, we have,
\begin{equation}\label{unlabelled trans weight}
    w(C_n\rightarrow C_{n+1})=\Psi(C_{n+1}|C_n)\ w(\lc{C}_{n}\rightarrow\lc{C}_{n+1}),
\end{equation}
where for any pair of representatives with $\lc{C}_n$ a down-set in $\lc{C}_{n+1}$.

Now we consider the couplings $(t_0,t_1,\cdots)$ in some detail. First, note the a set of couplings $(t_k)_{k\geq 0}$ provides a projective parameterisation, since any two sequences related by an overall positive factor, $(t_k)_{k\geq 0}=c(t'_k)_{k\geq 0}$, give rise to the same set of transition probabilities (though they give rise to different weights) \cite{Martin:2000js}. Second, we remark below on a meaningful interpretation of the couplings which will be useful for us.

\begin{remark}\label{rem CSG by proto-past}
  At stage $n$ of the CSG process, a subset $\lc{R}\subseteq\lc{C}_n$ is selected with relative probability $t_{|\lc{R}|}$ and the new element $n+1$ is put above all elements which are below or equal to some element in $\lc{R}$. We say that $\lc{R}$ is the \emph{proto-past} from which the past of $n+1$ is constructed. This interpretation is reflected in the form of the transition probabilities given in \eqref{transprob1}: the numerator is a sum over all proto-pasts which contribute to a particular transition $\lc{C}_n\rightarrow\lc{C}_{n+1}$, while the denominator sums over the proto-pasts which contribute to all possible transitions from the parent $\lc{C}_n$. An illustration is shown in Fig.\ref{proto_past_illst}.
\end{remark}
\begin{figure}[htpb]
  \centering
	\includegraphics[width=0.7\textwidth]{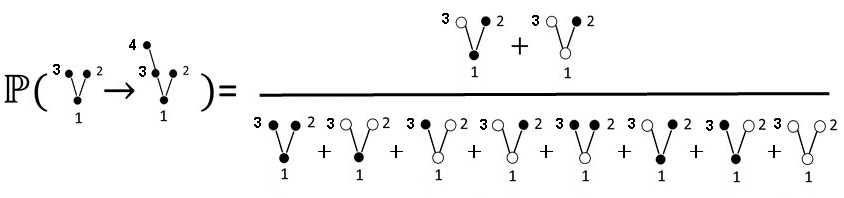}
	\caption[Pictorial representation of the CSG transition probabilities]{A pictorial representation of the transition probabilities of \eqref{transprob1}. Each diagram on the right hand side represents a proto-past $\lc{R}$ (shown in white) and contributes a factor of $t_{|\lc{R}|}$. The numerator is a sum over all proto-pasts which contribute to the transition on the left hand side. The denominator sums over all possible proto-pasts.}
	\label{proto_past_illst}
\end{figure}

Thus, $t_0$ is the relative probability of the new element to be minimal in the poset, and $t_1$ is the relative probability of the new element to cover exactly one element in the poset (though it may have more elements in its past). Although there is an infinite sequence of couplings, only a finite number of them is needed to compute each probability since the couplings $t_k$ with $k\geq n$ are ``inactive'' at stage $n$ (one cannot choose a subset with a cardinality greater than that of the set!).

\emph{Originary} CSG models are those with $t_0=0$ and $t_1>0$. These models only grow posets with a unique minimal element (an ``origin'').\footnote{The original formulation of the CSG models required that $t_0>0$ \cite{Rideout:1999ub}, but since then variations including the originary models have been widely studied, see for example \cite{Martin:2000js, Dowker:2005gj, Varadarajan:2005gg}.} More care is needed in defining CSG models with $t_k=0$ for all $k\leq N$ for some $N>0$, since these models require a choice of initial conditions beyond the trivial $\mathbb{P}(\lc{C}_1)=1$. Such models are not common in the physics literature, perhaps because a choice of one set of initial conditions over another would require additional physical motivation. For our purposes, we will define them as follows. Given some $N>0$, a CSG model with $t_k=0$ for all $k < N$ and $t_N>0$ is given by (i) the set of transition probabilities \eqref{transprob1} restricted to $n\geq N$ and (ii) a probability distribution on the posets of cardinality $N$,  \textit{i.e.} a complete set of probabilities $\mathbb{P}(\lc{C}_N)$ satisfying the normalisation condition \eqref{sum rule 2}. In these models, it is meaningless to ask what is the probability of a poset $\lc{C}_n$ with $n<N$.

The following CSG models will be important for us:
\begin{itemize}
\item The Forest Models: $t_0>0,t_1>0$, $t_k=0 \ \forall \ k>1$. This is a 1-parameter family of models which only grows forests with roots as minimal elements, \textit{i.e.} $\mathbb{P}(C_n)=0$ if $C_n$ is not a forest. The probability of growing a forest $C_{n}$ with $\tau$ trees is given by ,\begin{equation}\label{forest_probability} \mathbb{P}(C_n)=\Psi(C_n)\frac{t_0^{\tau-1}t_1^{n-\tau}}{\prod_{x=1}^{n-1}(t_0+xt_1)}.\end{equation}

\item The Tree Model: $t_1\not=0$, $t_k=0 \ \forall \ k\not=1$. Due to the projective nature of the $t_k$ parameters, all values of $t_1$ are equivalent as they give rise to the same transition probabilities. This model only grows trees, with roots as minimal elements, and can be seen as the $t_0\rightarrow 0$ limit of the forest models. The probability of growing a tree $C_n$ is given by, \begin{equation}\mathbb{P}(C_n)=\frac{\Psi(C_n)}{(n-1)!}.\end{equation}

\item Transitive Percolation: $\frac{t_{k+1}}{t_{k}}=t \ \forall \ k, t>0$. This 1-parameter family is also known as the model of \emph{random graph orders} in the mathematics literature \cite{Alon:1994}. Defining the Transitive Percolation parameters,
\begin{equation}\label{pqdefinition}
p=\frac{t}{1+t}, \ \ q=1-p, 
\end{equation}
the probabilities \eqref{transprob1} and \eqref{eqtn0606231} can be recast as,
\begin{equation}\label{TP_probability}\begin{split}
&\mathbb{P}(\tilde{C}_n \rightarrow \tilde{C}_{n+1}) =p^mq^{n-\varpi},\\
&\mathbb{P}(C_n)=\Psi(C_n)p^Lq^{\binom{n}{2}-R},\end{split}
\end{equation}
where $L$ and $R$ are the total number of links and relations in $C_n$, respectively.
An interpretation of this form is that each new element forms a relation with each already-existing element with probability $p$ independently and then the transitive closure is taken to obtain $\lc{C}_{n+1}$. This reflects the ``local'' nature of Transitive Percolation and all other CSG models can be seen as non-local generalisations of it in which the probability of forming a relation with a given element depends on whether or not a relation is formed with each of the other elements.

\item The Dust Model: $t_0>0, t_k=0 \forall k>0$. This model generates antichains (posets in which none of the elements are related) with unit probability. It is the $t_1\rightarrow 0$ limit of the forest models, and the $t\rightarrow 0$ (or $p\rightarrow 0$) limit of transitive percolation.
\end{itemize}


\subsection{Hopf algebras}\label{sec hopf}

Let $\mathbb{K}$ be a field of characteristic $0$; the reader will lose nothing in taking $\mathbb{K}=\mathbb{R}$ or $\mathbb{K}=\mathbb{C}$.

In order to define a Hopf algebra, let us first review the notion of algebra in a language most suited to what will follow.

An \emph{algebra} over $\mathbb{K}$ is a vector space $A$ over $\mathbb{K}$ along with linear maps $m:A\otimes A\rightarrow A$, called the \emph{product}, and $\eta:\mathbb{K}\rightarrow A$, called the \emph{unit}, with the properties
\begin{itemize}
\item (associativity) $m\circ (m\otimes \text{id})= m \circ (\text{id}\otimes m)$ where $\text{id}$ is the identity map.
\item (unitality) $m\circ  (\eta\otimes \text{id}) = \text{id} = m\circ  (\text{id} \otimes \eta)$ using the canonical isomorphisms between $\mathbb{K}\otimes A$, $A$ and $A\otimes \mathbb{K}$.
\end{itemize}

For the reader who is more familiar with the usual undergrad abstract algebra take on an algebra, note that if we rewrite $m(a,b)$ as $a\cdot b$ then the first property is $(a\cdot b)\cdot c = a\cdot(b\cdot c)$ which is the associativity of $\cdot$.  The unital map $\eta$ relates to the unit in the more naive sense in that $\eta(1)\in A$ is the unit.  The property above tells us this because the canonical isomorphism between $\mathbb{K}\otimes A$ and $A$ is the one taking $1\otimes a$ to $a$, and so the unital property tells us $\eta(1)\cdot a = a = a \cdot \eta(1)$ as expected of a unit.  The final thing about this formulation which might be unexpected from the perspective of undergrad abstract algebra is how the product has domain $A\otimes A$ rather than $A\times A$.  The point here is that bilinear maps on $A\times A$ correspond to linear maps on $A\otimes A$ via the universal property of tensor product.  For this formulation using $A\otimes A$ is better, but the same information is being carried either way.

To get a coalgebra we take analogous maps and properties as in the definition of algebra, but reverse their directions.

A \emph{coalgebra} over $\mathbb{K}$ is a vector space $C$ over $\mathbb{K}$ along with linear maps $\Delta:C\rightarrow C\otimes C$, called the \emph{coproduct}, and $\epsilon:C\rightarrow \mathbb{K}$, called the \emph{counit}, with the properties
\begin{itemize}
\item (coassociativity) $(\Delta\otimes \text{id})\circ \Delta= (\text{id}\otimes \Delta)\circ \Delta$ where again $\text{id}$ is the identity map.
\item (counitality) $(\epsilon\otimes \text{id})\circ  \Delta = \text{id} = (\text{id} \otimes \epsilon)\circ  \Delta$ again using the canonical isomorphisms between $\mathbb{K}\otimes C$, $C$ and $C\otimes \mathbb{K}$.
\end{itemize}

All these identities can be written quite insighfully in terms of commutative diagrams, and the duality which gives the coalgebra is particularly clear in that way.  In the interests of space we will leave the commutative diagrams to the references.  One suitable nice reference with a combinatorial focus is \cite{GRhopf}.

A $\mathbb{K}$-linear map $f:A\rightarrow B$ between two algebras $A$ and $B$ is an \emph{algebra homomorphism} (or \emph{algebra morphism}) if $f\circ m_A = m_B\circ (f\otimes f)$ and $f \circ \eta_A = \eta_B$.  Again these identities correspond exactly to what one expects from the more typical undergrad take on algebras and as expected flipping all arrows gives the analogous coalgebra notion. A map $g:C\rightarrow D$ between two coalgebras $C$ and $D$ is a \emph{coalgebra homomorphism} (or \emph{coalgebra morphism}) if $\Delta_D \circ g  = (g \otimes g)\circ \Delta_C$ and $\epsilon_D\circ g = \epsilon_C$.

A bialgebra is simultaneously and compatibly both an algebra and a coalgebra in the following sense.  Let $B$ be both an algebra over $\mathbb{K}$ and a coalgebra over $\mathbb{K}$ with the property that $\Delta$ and $\epsilon$ are both algebra homomorphims, or equivalently that $m$ and $\eta$ are both coalgebra homomorphims.  Then we say $B$ is a \emph{bialgebra}.  The equivalence between  $\Delta$ and $\epsilon$ being algebra homomorphims and $m$ and $\eta$ being coalgebra homomorphims, is easy to check as the four necessary properties in each case end up being the same.  This can be found in any standard reference, see for instance Proposition 1.3.5 of \cite{GRhopf}.

To get a Hopf algebra we need one more map.  First, if $A$ is an algebra, $C$ a coalgebra, and $f,g:C\rightarrow A$ linear maps between them, then the \emph{convolution product} of $f$ and $g$ is
\[
f\star g = m \circ (f\otimes g)\circ \Delta
.\]
Then a bialgebra $B$ is a \emph{Hopf algebra} if there exists a linear map $S:B\rightarrow B$, called the \emph{antipode}, such that $S\star \text{id} = \eta \circ \epsilon =\text{id}\star S$.

For an algebra, coalgebra, or bialgebra, if the underlying vector space is graded and the defining maps are also graded then we say the algebra, coalgebra, or bialgebra is \emph{graded}.  Our examples will always be combinatorial and so have a grading coming from the notion of size on the combinatorial objects -- for us the number of elements in a poset or the number of vertices of a rooted tree.  Additionally, our combinatorial contexts will always have a unique empty object of size $0$ and so the $0$-graded piece will always be simply a copy of $\mathbb{K}$.  A convenient result for us is the following, see Proposition 1.4.16 of \cite{GRhopf} for a proof.

\begin{prop}\label{prop auto antipode}
  Let $B$ be a bialgebra that is graded and has the $0$-graded piece\footnote{When the $B$ has $0$-graded piece isomorphic to $\mathbb{K}$ then $B$ is said to be \emph{connected}.  However, this notion of connected bialgebra is quite different from the notion of connected poset or connected graph, and this is the only place we need it, so to avoid confusion we will avoid using the language of connectivity for bialgebras.} isomorphic to $\mathbb{K}$. Then $B$ is a Hopf algebra with antipode recursively defined.
\end{prop}

Given an algebra $A$, a subset $A'$ of $A$ is a \emph{subalgebra} if it is a subspace of $A$ and is closed under $m$ and $\eta$, that is $A'$ is a subspace and for all $a,b\in A'$, $m(a,b)\in A'$ and for all $k\in \mathbb{K}$, $\eta(k)\in A'$.
Likewise given a bialgebra $B$, a subset $B'$ of $B$ is a \emph{subbialgebra} of $B$ if it is a subspace of $B$ and is closed under $m$, $\eta$, $\Delta$, and $\epsilon$.  Note that closure under $\epsilon$ is trivial -- there is nothing to check.  In our combinatorial examples, the subbialgebras we're interested in will be built of homogenous generators and also have a unique generator of degree 0 and so we will be in the situation where $B$ and $B'$ will both be graded and connected and hence so both be Hopf algebras.  Furthermore, when both are graded and connected with the same grading, the recursive expressions for the antipodes mentioned in Proposition~\ref{prop auto antipode} agree giving that $B'$ is a \emph{subHopf algebra} of $B$.

We will be particularly interested in the situation where we have a graded connected Hopf algebra $H$ and we have a subset $H'$ of $H$ that is a subalgebra by construction, but we will want to know when $H'$ is a subHopf algebra.  By the observations above, the only thing to check is that $H'$ is closed under $\Delta$.  That is, we will need to check that for any $h\in H'$,  $\Delta(h)\in H'\otimes H'$.

\medskip

Now we are ready to define the primary Hopf algebra of interest to us.  Let $\mathcal{P}$ be the set of (unlabelled) posets and let $\mathcal{C}\subseteq \mathcal{P}$ be the set of connected posets.  As a vector space, the Hopf algebra of posets is $\text{span}_\mathbb{K}(\mathcal{P})$.  We make this into an algebra by taking disjoint union as the product, and the empty poset as the image of $\eta(1)$. Identifying monomials with the disjoint union of their elements, this is equivalent to saying that as an algebra the Hopf algebra of posets is $\mathbb{K}[\mathcal{C}]$.  We will take the counit to be $\epsilon(P)=0$ for nonempty posets $P$, $\epsilon(\emptyset)=1$ and extended linearly; this is an algebra homomorphism.  Note that $\mathbb{K}[\mathcal{C}]$ is graded by the size function on posets and with this grading has $\text{span}_\mathbb{K}(\{\emptyset\}) \cong \mathbb{K}$ as the $0$-graded piece.  All the maps, including $\epsilon$ are graded, and so it remains to define a compatible graded coproduct on $\mathbb{K}[\mathcal{C}]$.

\begin{definition}\label{def coproduct on posets}
  Define the coproduct on $\mathbb{K}[\mathcal{C}]$ to be
  \[
  \Delta(P) = \sum_{U \text{ up-set of } P} U \otimes (P\setminus U) = \sum_{D \text{ down-set of } P} (P\setminus D) \otimes D
  \]
  for $P\in \mathcal{P}$ and extended linearly.
\end{definition}
Observe that the two sums of the definition do agree since the complement of an up-set is a down-set and vice versa.  Additionally, $\Delta$ is graded since each element of $P$ appears on exactly one side of each term in the sum.

One can check directly that with this coproduct $\mathbb{K}[\mathcal{C}]$ is a Hopf algebra, see Section 13.1 of \cite{AMmonoidal} for details.

Since the interplay between unlabelled and naturally labelled will be important later, it is worth being a bit more explicit about how to understand labellings and the coproduct.  As defined the coproduct is for unlabelled posets, however an unlabelled poset is an isomorphism class and the sum over up-sets (or down-sets) should be interpreted as a sum over up-sets (or down-sets) in any representative of the class with the resulting summands only subsequently taken up to isomorphism.  In particular, if a poset $P$ has two isomorphic copies of the same up-set, then both of them contribute to the sum.  This means that the coproduct of particular posets may contain multiplicities, for example,
$$\begin{tikzpicture}
\node[align=left] at (-2.35,-0.4) {$\Delta($};
\filldraw[black] (-2,-0.25) circle (1.5pt) node[anchor=west]{};
\filldraw[black] (-2,-0.5) circle (1.5pt) node[anchor=west]{};
\draw[black, thick] (-2,-0.25) -- (-2,-0.5);
\filldraw[black] (-1.8,-0.25) circle (1.5pt) node[anchor=west]{};
\filldraw[black] (-1.8,-0.5) circle (1.5pt) node[anchor=west]{};
\draw[black, thick] (-1.8,-0.25) -- (-1.8,-0.5);
\node[align=left] at (-1.1,-0.4) {$)=1\otimes$};
\filldraw[black] (-.3,-0.25) circle (1.5pt) node[anchor=west]{};
\filldraw[black] (-.3,-0.5) circle (1.5pt) node[anchor=west]{};
\draw[black, thick] (-.3,-0.25) -- (-.3,-0.5);
\filldraw[black] (-0.1,-0.25) circle (1.5pt) node[anchor=west]{};
\filldraw[black] (-0.1,-0.5) circle (1.5pt) node[anchor=west]{};
\draw[black, thick] (-0.1,-0.25) -- (-0.1,-0.5);

\node[align=left] at (0.3,-0.4) {$+2$};
\filldraw[black] (.7,-0.4) circle (1.5pt) node[anchor=west]{};
\node[align=left] at (1,-0.4) {$\otimes$};
\filldraw[black] (1.3,-0.5) circle (1.5pt) node[anchor=west]{};
\draw[black, thick] (1.3,-0.25) -- (1.3,-0.5);
\filldraw[black] (1.3,-0.25) circle (1.5pt) node[anchor=west]{};
\filldraw[black] (1.5,-0.5) circle (1.5pt) node[anchor=west]{};

\node[align=left] at (1.9,-0.4) {$+2$};
\filldraw[black] (2.3,-0.25) circle (1.5pt) node[anchor=west]{};
\filldraw[black] (2.3,-0.5) circle (1.5pt) node[anchor=west]{};
\node[align=left] at (2.6,-0.4) {$\otimes$};
\draw[black, thick] (2.3,-0.25) -- (2.3,-0.5);
\draw[black, thick] (2.9,-0.25) -- (2.9,-0.5);
\filldraw[black] (2.9,-0.25) circle (1.5pt) node[anchor=west]{};
\filldraw[black] (2.9,-0.5) circle (1.5pt);
\node[align=left] at (3.2,-0.4) {$+$};
\filldraw[black] (3.5,-0.4) circle (1.5pt) node[anchor=west]{};
\filldraw[black] (3.7,-0.4) circle (1.5pt) node[anchor=west]{};
\node[align=left] at (4,-0.4) {$\otimes$};
\filldraw[black] (4.3,-0.4) circle (1.5pt) node[anchor=west]{};
\filldraw[black] (4.5,-0.4) circle (1.5pt) node[anchor=west]{};

\node[align=left] at (4.9,-0.4) {$+2$};
\node[align=left] at (5.8,-0.4) {$\otimes$};
\filldraw[black] (5.3,-0.5) circle (1.5pt) node[anchor=west]{};
\draw[black, thick] (5.3,-0.25) -- (5.3,-0.5);
\filldraw[black] (5.3,-0.25) circle (1.5pt) node[anchor=west]{};
\filldraw[black] (5.5,-0.5) circle (1.5pt) node[anchor=west]{};
\filldraw[black] (6.1,-0.4) circle (1.5pt) node[anchor=west]{};
\node[align=left] at (6.4,-0.4) {$+$};

\node[align=left] at (7.4,-0.4) {$\otimes1,$};
\filldraw[black] (6.7,-0.5) circle (1.5pt) node[anchor=west]{};
\draw[black, thick] (6.7,-0.25) -- (6.7,-0.5);
\draw[black, thick] (7,-0.25) -- (7,-0.5);
\filldraw[black] (6.7,-0.25) circle (1.5pt) node[anchor=west]{};
\filldraw[black] (7,-0.5) circle (1.5pt) node[anchor=west]{};
\filldraw[black] (7,-0.25) circle (1.5pt) node[anchor=west]{};
\end{tikzpicture}$$
where we denote the empty set by 1.
It will be useful later to have language for the terms of the sum before forgetting the labels.  Define a \emph{labelled cut} as a pair $(\lc{C}, \lc{\mathcal{U}}\otimes\lc{\mathcal{D}})$ consisting of a naturally labelled poset $\lc{C}$ and a partition of it into two increasingly labelled posets, a downset $\lc{\mathcal{D}}$ and an upset $\lc{\mathcal{U}}$. Examples are shown in Fig.\ref{splits_example}. For any fixed naturally labelled representative $\lc{P}$ of a poset $P$, the terms of $\Delta(P)$ are exactly the second entries in the labelled cuts of $\lc{P}$ after forgetting their labellings.

\begin{figure}[htpb]
  \centering
	\begin{tikzpicture}
\node[align=left] at (-2.5,-0.4) {$\bigg($};
\filldraw[black] (-2,0) circle (2pt) node[anchor=west]{2};
\filldraw[black] (-1.2,0) circle (2pt) node[anchor=west]{4};
\filldraw[black] (-2,-0.75) circle (2pt) node[anchor=west]{1};
\filldraw[black] (-1.2,-0.75) circle (2pt) node[anchor=west]{3};
\draw[black, thick] (-2,0) -- (-2,-0.75);
\draw[black, thick] (-1.2,0) -- (-1.2,-0.75);
\node[align=left] at (-0.5,-0.7) {,};

\filldraw[black] (0,0) circle (2pt) node[anchor=west]{2};
\filldraw[black] (0.8,0) circle (2pt) node[anchor=west]{4};
\filldraw[black] (0,-0.75) circle (2pt) node[anchor=west]{1};
\filldraw[black] (0.8,-0.75) circle (2pt) node[anchor=west]{3};
\draw[black, thick] (0,0) -- (0,-0.75);
\draw[black, thick] (0.8,0) -- (0.8,-0.75);
\node[align=left] at (0.5,-0.4) {$\otimes$};
\node[align=left] at (1.5,-0.4) {$\bigg)$};

\node[align=left] at (2.5,-0.4) {$\bigg($};
\filldraw[black] (3,0) circle (2pt) node[anchor=west]{2};
\filldraw[black] (3.8,0) circle (2pt) node[anchor=west]{4};
\filldraw[black] (3,-0.75) circle (2pt) node[anchor=west]{1};
\filldraw[black] (3.8,-0.75) circle (2pt) node[anchor=west]{3};
\draw[black, thick] (3,0) -- (3,-0.75);
\draw[black, thick] (3.8,0) -- (3.8,-0.75);
\node[align=left] at (4.5,-0.7) {,};

\filldraw[black] (5,0) circle (2pt) node[anchor=west]{4};
\filldraw[black] (5.8,0) circle (2pt) node[anchor=west]{2};
\filldraw[black] (5,-0.75) circle (2pt) node[anchor=west]{3};
\filldraw[black] (5.8,-0.75) circle (2pt) node[anchor=west]{1};
\draw[black, thick] (5,0) -- (5,-0.75);
\draw[black, thick] (5.8,0) -- (5.8,-0.75);
\node[align=left] at (5.5,-0.4) {$\otimes$};
\node[align=left] at (6.5,-0.4) {$\bigg)$};

\node[align=left] at (7.5,-0.4) {$\bigg($};
\filldraw[black] (8,0) circle (2pt) node[anchor=west]{2};
\filldraw[black] (8,-0.75) circle (2pt) node[anchor=west]{1};
\filldraw[black] (8,1.5) circle (2pt) node[anchor=west]{4};
\filldraw[black] (8,0.75) circle (2pt) node[anchor=west]{3};
\draw[black, thick] (8,0) -- (8,-0.75);
\draw[black, thick] (8,0) -- (8,.75);
\draw[black, thick] (8,0) -- (8,1.5);
\node[align=left] at (8.5,-0.7) {,};

\filldraw[black] (9,0) circle (2pt) node[anchor=west]{4};
\filldraw[black] (9.8,0) circle (2pt) node[anchor=west]{2};
\filldraw[black] (9,-0.75) circle (2pt) node[anchor=west]{3};
\filldraw[black] (9.8,-0.75) circle (2pt) node[anchor=west]{1};
\draw[black, thick] (9,0) -- (9,-0.75);
\draw[black, thick] (9.8,0) -- (9.8,-0.75);
\node[align=left] at (9.5,-0.4) {$\otimes$};
\node[align=left] at (10.5,-0.4) {$\bigg)$};
\end{tikzpicture}
	\caption{Three distinct labelled cuts.}
	\label{splits_example}
\end{figure}

\medskip

When we look at the special case of posets which are forests, we get another well-known Hopf algebra.  Let $\mathcal{H}_{CK}$ be the subset of $\mathbb{K}[\mathcal{C}]$ given by the span of forests viewed as posets with the roots as maximal elements.  Then $\mathcal{H}_{CK}$ is a subHopf algebra, called the \emph{Connes--Kreimer Hopf algebra of rooted trees} \cite{Khopf, ckI, ckII}.

The Connes--Kreimer Hopf algebra of rooted trees is used in the Hopf algebraic formulation of renormalization in quantum field theory.  In this context the trees give the insertion structure of subdivergent Feynman diagrams in a larger Feynman diagram, and the Hopf algebra structure gives the Zimmerman forest formula and hence can encode BPHZ renormalization.  Overlapping subdivergences are represented by sums of trees.

Note that we could alternately represent the Connes--Kreimer Hopf algebra inside $\mathbb{K}[\mathcal{C}]$ with roots as minimal elements, however, we will always want to use the form with roots as maximal elements, because the recursive structure of $\mathcal{H}_{CK}$ comes from building new trees by adding a new root to a forest, and, inspired by the CSG model, we will always build by adding new maximal elements.

We encode the add-a-root construction in the following definition: given $f$ a forest, define $B_+(f)$ to be the rooted tree obtained by adding a new root $r$ and letting the children of $r$ be the roots of the trees in $f$.  Extend $B_+$ linearly to $\mathcal{H}_{CK}$.

Recursive equations using $B_+$ in $\mathcal{H}_{CK}$ are the combinatorial avatar of the \emph{Dyson-Schwinger equations} of quantum field theory \cite{Ybook}.  An important case of such equations is those of the form
\begin{equation}\label{eq combDSE simple tree}
X(x) = xB_+(f(X(x))
\end{equation}
where $f$ is a formal power series with constant term equal to $1$. 
This equation has a unique solution $X(x)\in \mathcal{H}_{CK}[x]$ defined recursively; see Proposition 2 of \cite{FoissyDyson}.  In the pure combinatorics context solutions to equations of the form \eqref{eq combDSE simple tree} are sometimes known as \emph{simple tree classes} following \cite{MMsimple}. 
The equation can also be written in the slightly different from
\[
Y(x) = 1+xB_+(g(Y(x))) 
\]
but after the substitution $X(x) = Y(x)-1$ and $f(z) = g(z+1)$ the only difference is in whether or not the solution includes a constant term, so it is mostly a matter of taste whether the base case of $1$ is inside the $B_+$ (by the condition on $f$) or outside the $B_+$ (by the explicit $1$).  However, one reason to avoid the constant term in the solution is that then the composition inside $B_+$ is not transparently a well-defined composition of formal power series; it is only the particular shape of $g$, having come from $f$, that makes this composition valid.  

Foissy characterized when the algebra generated by the coefficients of $X(x)$ is a Hopf subalgebra of $\mathcal{H}_{CK}$.  Specifically:

\begin{prop}[Theorem 4 of \cite{FoissyDyson}]
\label{prop Foissy_characterization_DSE}
Let $f \in \mathbb{K}[[x]]$ such that $f(0) = 1$ and let $X(x)=\sum_{n\geq 1} a_nx^n$ with $a_n\in \mathcal{H}_{CK}$ be the unique solution to $X(x) = xB_+(f(X(x)))$.  The following are equivalent:
\begin{enumerate}
    \item $\mathbb{K}[a_1, a_2, a_3, \ldots]$ is a subHopf algebra of $\mathcal{H}_{CK}$
    \item There exists $(\alpha, \beta) \in \mathbb{K}^2$ such that $(1 - \alpha\beta x)f'(x) = \alpha f(x)$ 
    \item\label{clause prop Foissy_characterization_DSE} There exists $(\alpha, \beta) \in \mathbb{K}^2$ such that 
    \begin{enumerate}
        \item $f(x) = 1$ if $\alpha = 0$
        \item $f(x) =e^{\alpha x}$ if $\beta = 0$
        \item $f(x) = (1 - \alpha \beta x)^{-\frac{1}{\beta}}$ if $\alpha \beta \neq 0$.
    \end{enumerate}
\end{enumerate}
\end{prop}

There's another important example of a Hopf algebra which can be built out of rooted trees but which does not come out of Proposition~\ref{prop Foissy_characterization_DSE}.  Usually this is formulated as a subHopf algebra of $\mathcal{H}_{CK}$, however, it is built by adding leaves, rather than adding roots, and so in the spirit that we always build upwards in this paper, we will instead define it as a subHopf algebra of $\mathbb{K}[\mathcal{C}]$ that consists of trees with minimal elements as roots, in contrast to how $\mathcal{H}_{CK}$ is built from trees with maximal elements as roots.

For a rooted tree $t$ (as a poset, with the root as the minimal element), define the natural growth operator $N(t)$ to be the sum of all trees obtained by adding a new leaf to a vertex of $t$, and extend $N$ linearly to the span of trees. Using $N$ we define,

\begin{equation}\label{CM generators}
    \delta_i = \begin{cases}
      \bullet& i=1\\
      N(\delta_{i-1})  & i>1.\\
    \end{cases} 
\end{equation}  Then $\mathcal{H}_{CM} = \mathbb{K}[\delta_1, \delta_2, \delta_3, \ldots]$ is a subHopf algebra of $\mathbb{K}[\mathcal{C}]$ and is called the \emph{Connes--Moscovici Hopf algebra} \cite{cmhopf, ck0}.
Note that by construction, the coefficient of a rooted tree in $\delta_i$ is the number of natural labellings of the tree.  We can draw the first few terms as follows,
\[
\delta_1 = \bullet \quad \delta_2 = \includegraphics{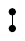} \quad \delta_3 = \includegraphics{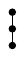} + \includegraphics{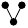}  \qquad \delta_4 = \includegraphics{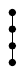} + \includegraphics{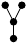} + 3\includegraphics{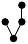} + \includegraphics{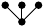} \ .
\]
In contrast, if we take the Dyson-Schwinger equation with $f(x)= (1-x)^{-1}$ ($\alpha=\beta=1$ in Proposition~\ref{prop Foissy_characterization_DSE}), then we get the expansion,
\[
a_1 = \bullet \quad a_2 = \includegraphics{T2} \quad a_3 = \includegraphics{T3a} + \includegraphics{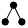}  \qquad a_4 = \includegraphics{T4a} + \includegraphics{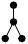} + 2\includegraphics{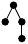} + \includegraphics{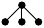}
\]
since the coefficents in this case count the number of plane embeddings.

\section{General set up}\label{sec set up}

The similarities and differences between $B_+$ and $N$ as operations for growing rooted trees along with the fact that both yield subHopf algebras has long suggested to researchers in the area to try to bring either the operations or the sums of trees they generate into a common framework.  However there has been no solution so that the Hopf structure comes in both cases from a result with the flavour of Proposition~\ref{prop Foissy_characterization_DSE}.  The CSG models are also built by a growing process, and the main insight of the present section is that moving to the poset context, and in both cases taking $B_+$ and $N$ as growing upwards, so having them involve different realizations of rooted trees within posets, finally permits this common framework.

We are not able, however, to give a generalization of Proposition~\ref{prop Foissy_characterization_DSE} to our full framework.  Rather, our main result, see Section~\ref{sec main result}, tackles the case of CSG models.  We none the less find our general framework interesting for its unification of combinatorial DSEs for trees, the Connes--Moscovici Hopf algebra, and CSG models into one framework.  In order to allow our models to be built simply by growing, without needing to count how many posets are grown, we will be building the un-normalised version of the CSG models here, as described in \eqref{weight}.

First we need an operation which is general enough to allow all of these growth models. Recall from definition \ref{def down set}, that given some $S\subseteq P$, $D(S)$ denotes the down-set generated by $S$.
\begin{definition}\label{def B_S}
  For a poset $P$ and a subset $S$ of elements of $P$, define $B_S(P)$ to be the poset obtained from $P$ by adding one new element which is larger than the elements of $D(S)$ and incomparable with all other elements of $P$.
\end{definition}
For example, $$ B_{\includegraphics{T2}}(\includegraphics{T3b})=\includegraphics{T4c}.$$
Here we are working with unlabelled posets, so we have not specified a label for the new element.  Alternately, in the naturally labelled context, the new element could be given the new label $|P|+1$, maintaining that the labelling is a natural labelling.

Note that if $P$ is a forest with roots as maximal elements then $B_P(P) = B_+(P)$, while if $P$ is a tree with roots as minimal elements then $\sum_{a\in P} B_{\{a\}}(P) = N(P)$.

From the CSG perspective the set $S$ in Definition~\ref{def B_S} is the proto-past of the to-be-added new element, see Remark~\ref{rem CSG by proto-past}.  Consequently, we may have $S\neq S'$ but $B_S(P)=B_{S'}(P)$ as two different proto-pasts may generate the same downset.  This can even happen when $S$ and $S'$ have the same size.  For example if $P=\{a<b<c\}$ then $B_{\{a,c\}}(P)= B_{\{b,c\}}(P)=\{a<b<c<d\}$ where we let the new element be $d$.

\begin{definition}\label{def M}
  Define $M_{\substack{s_0, s_1, \ldots\\t_0, t_1, \ldots}}$ to be the linear map that takes a poset $P$ to
  \[
  M_{\substack{s_0, s_1, \ldots\\t_0, t_1, \ldots}}(P) = \sum_{S\subseteq P} t_{|S|}s_{|P|-|S|}B_S(P)
  \]
  and extended linearly to $\text{span}_\mathbb{K}(\mathcal{P}) = \mathbb{K}[\mathcal{C}]$. 
\end{definition}
 
$M_{\substack{s_0, s_1, \ldots\\t_0, t_1, \ldots}}$  is the operator that puts a factor $t_ks_{n-k}$ when growing an $n$ element poset by putting a new descendant of a subset of size $k$ (cf. Remark~\ref{rem CSG by proto-past}). 

For the CSG context we will take $t_i\in \mathbb{R}$, $t_i\geq 0$.  However, for the purposes of the general framework we could just as well treat the $s_i$ and $t_i$ as indeterminates and upgrade $M_{\substack{s_0, s_1, \ldots\\t_0, t_1, \ldots}}$ to act on $\mathbb{K}[\mathcal{C}][s_0, s_1, \ldots, t_0, t_1, \ldots]$ by acting on the coefficients of each monomial in the $s_i$ and $t_j$.

The operator $M_{\substack{s_0, s_1, \ldots\\t_0, t_1, \ldots}}$ is the simultaneous generalization of all our growth operators to date.
\begin{lemma}\label{lemma M}
\ 
    \begin{itemize}
        \item For $f$ a forest with roots as maximal elements,
        \[
            B_+(f) = M_{\substack{1, 0, 0, \ldots\\1, 1, 1, \ldots}}(f),
        \]
        that is, $B_+$ is $M$ with $s_0=1$, $s_i=0$ for $i>0$, and all $t_j=1$.
        \item For $t$ a tree with the root as the minimal element,
        \[
            N(t) = M_{\substack{1, 1, 1,1, \ldots\\0, 1, 0, 0, \ldots}}(t),
        \]
        that is, $N$ is $M$ with $t_1=1$, $t_i=0$ for $i\neq 1$ and all $s_j=1$.
         \item For $C_n$ a poset with $n$ elements,
        \[
\sum_{C_{n+1}}w(C_{n}\rightarrow C_{n+1})C_{n+1}=M_{\substack{1, 1, 1, \ldots\\t_0, t_1, \ldots}}(C_n)
        \]
        where the sum is over all children $C_{n+1}$ of $C_{n}$ and the transition weight $w(C_n\rightarrow C_{n+1})$ is as given in \eqref{unlabelled trans weight}. That is, the CSG growth operation with couplings $t_0,t_1,\ldots$ is $M$ with all $s_i=1$.
    \end{itemize}
\end{lemma}

\begin{proof}
\ 
    \begin{itemize}
        \item $s_0$ appears as a coefficient in $M_{\substack{s_0, s_1, \ldots\\t_0, t_1, \ldots}}(P)$ when the proto-past $S$ consists of all elements of $P$ and $s_i$ for $i>0$ appears when $S$ consists of a proper subset of elements of $P$.  Therefore setting $s_0=1$ and $s_i=0$ for $i>0$ means we are only considering the whole input poset as a possible proto-past, \textit{i.e.} the new element added is above all existing elements.  Given that the input is a forest $f$ with roots as maximal elements,  the new element added is the new root of $B_+(f)$.  The $t_j$ play no role so are all set to $1$.
        \item $t_1$ appears as a coefficient in $M_{\substack{s_0, s_1, \ldots\\t_0, t_1, \ldots}}(P)$ when the proto-past $S$ consists of a single element of $P$ and $t_i$ for $i\neq 1$ appears when $S$ consists of some other number of elements of $P$.  Therefore setting $t_0=1$ and $t_i=0$ for $i>0$ means we are only considering single elements of $P$ as possible proto-pasts. Given that the input is a tree $t$ with the root as the minimal element, this means that the allowed proto-pasts correspond to the various ways of adding a new leaf to $t$, \textit{i.e.} the various ways of growing through $N(t)$. The $s_j$ play no role so are all set to $1$.
        \item With the $s_j$ set to $1$, $M$ is directly the growth operator of the CSG model as described in Remark~\ref{rem CSG by proto-past}, but without the normalization (cf. equations \eqref{weight} and \eqref{unlabelled trans weight}).
    \end{itemize}
\end{proof}

We are interested in recursively building series of posets in $\mathbb{K}[\mathcal{C}][[x]]$ using the operator $M_{\substack{s_0, s_1, \ldots\\t_0, t_1, \ldots}}$, in such a way that the combinatorial Dyson-Schwinger equations on trees, the Connes-Moscovici Hopf algebra, and the CSG models can all be produced.  Specifically, we are interested in solutions to equations of the form
\begin{equation}\label{eq general set up}
  A(x) = b(x) + xM_{\substack{s_0, s_1, \ldots\\t_0, t_1, \ldots}}(f(A(x)))
\end{equation}
where $f(u)\in \mathbb{K}[[u]]$ is a formal power series with $f(0)=1$ and $b(x) \in \mathbb{K}[\mathcal{C}][x]$ is a polynomial in $x$ where $b(0)=0$, the coefficient of $x^k$ is homogeneous of degree $k$ in $\mathbb{K}[\mathcal{C}]$, and degree bound,
\begin{equation}
\deg(b(x)) =\min\{k: t_k \neq 0\} + \min\{k: s_k \neq 0\}. 
\end{equation}

Note that in contrast to the set up from Proposition~\ref{prop Foissy_characterization_DSE} we have the potential for base case terms outside the growth operator $M$ in the polynomial $b(x)$.  This is because in the case that some $s_i$ or $t_i$ vanish then $M$ may not generate any poset when applied to the constant term of $f$ and so we will need to include an external base case.  See the definition of originary models and the discussion that follows in Section~\ref{sec cst} for the analogous situation in that context.  The degree constraint on $b(x)$ avoids overlap between posets explicitly in $b(x)$ and those built by $M$ and also forces there to be something on which $M$ can operate.  However, since $b(0)=0$ we are also guaranteed that the composition $g(A(x))$ is well defined as a composition of formal power series.

\begin{lemma}\label{lemma unique soln}
    With hypotheses as above, \eqref{eq general set up} recursively defines a unique series $A(x)$.
\end{lemma}

\begin{proof}
    Write $A(x) = \sum_{n\geq 1}a_nx^n$ with the $a_n\in \mathbb{K}[\mathcal{C}]$ and write $b(x)=\sum_{n=0}^{d} b_nx^n$ with the $b_n\in \mathbb{K}[\mathcal{C}]$ and with $d=\deg(b(x))$.
    

    Suppose $d>0$ and $n\leq d$.  By the degree constraint on $b(x)$, we have $a_n=b_n$ as the other term in \eqref{eq general set up} does not contribute.  Now suppose $d=0$. Then $a_1 = M_{\substack{s_0, s_1, \ldots\\t_0, t_1, \ldots}}(f(0)) = s_0t_0$.  This gives the base case for our induction.

    Now suppose that for $i<n$ the $a_i$ are uniquely determined.  Consider $a_n$.  This is the coefficient of $x^n$ in $b(x) + xM_{\substack{s_0, s_1, \ldots\\t_0, t_1, \ldots}}(f(A(x)))$, but by the degree constraints on $b(x)$ this is equal to $M_{\substack{s_0, s_1, \ldots\\t_0, t_1, \ldots}}$ applied to the coefficient of $x^{n-1}$ in $f(A(x))$ for $n>d$.  However, the coefficient of $x^{n-1}$ in $f(A(x))$ depends only on the $a_i$ with $i<n$ and the coefficients of $f$. Inductively these coefficients are known and hence so is $a_n$.
\end{proof}

Alternately, for the reader who does not want to get their hands dirty, Lemma~\ref{lemma unique soln} can be proved analytically by the Banach fixed point theorem using the usual metric on the ring of formal power series.

Observe that \eqref{eq general set up} simultaneously generalizes all of the growth models we have discussed.  By Remark~\ref{rem CSG by proto-past}, CSG models can be obtained from \eqref{eq general set up} by taking $f(u)=1+u$, and $s_i=1$ for all $i$.  The combinatorial Dyson-Schwinger equations of Proposition~\ref{prop Foissy_characterization_DSE} where trees are taken as having roots as maximal elements can be obtained from \eqref{eq general set up} by taking $s_0=1$, $s_i=0$ for $i>0$, $t_j=1$ for all $j$, in which case $b(x)=0$.  The Connes-Moscovici Hopf algebra can be generated, where trees are taken as having roots as minimal elements, by taking $t_1=1$, $t_i=0$ for $i\neq 1$, all $s_j=1$, $b(x)= x\bullet$ and $f(u)=1+u$.  

\medskip

The question of interest to us regarding solutions to \eqref{eq general set up} is when the coefficients generate a subHopf algebra.  Specifically, if $A(x)=\sum_{n\geq 1}a_nx^n$ is the solution to \eqref{eq general set up}, what conditions on $b(x), f(x), s_i$ and $t_i$ must hold in order for $\mathbb{K}[a_1, a_2, \ldots] \subseteq \mathcal{P}$ to be a subHopf algebra of $\mathcal{P}$? Foissy's result, Proposition~\ref{prop Foissy_characterization_DSE}, answers one special case of that question.  The fact that the Connes-Moscovici Hopf algebra is a Hopf algebra gives another answer in a special case.
In the rest of this work, we answer this question in the special case of the CSG models.


\section{Main result}\label{sec main result}

Our main result characterises when certain solutions to \eqref{eq general set up} are subHopf.
\begin{thm}\label{thm for general set up} Let $A(x)=\sum_{n\geq 1} a_nx^n$ with $a_n\in \mathbb{K}[\mathcal{C}]$ be the unique solution to \eqref{eq general set up}.
\begin{enumerate}
    \item If $t_i=t_j\not=0 \ \forall \ i,j$ and $s_0\not=0, s_i=0 \ \forall \ i>0$, then $\mathbb{K}[a_1,a_2,a_3,\ldots]$ is a subHopf algebra of $\mathbb{K}[\mathcal{C}]$ if and only if $f(u)$ satisfies the conditions of Proposition~\ref{prop Foissy_characterization_DSE}, clause~{\normalfont(\ref{clause prop Foissy_characterization_DSE})}.
    \item\label{main thm clause} If $f(u)=1+u$, $\mathbb{R}\subseteq \mathbb{K}$, $s_i=s_j\not=0 \ \forall \ i,j$ and $t_i\geq 0\ \forall i$ then $\mathbb{K}[a_1,a_2,a_3,\ldots]$ is a subHopf algebra of $\mathbb{K}[\mathcal{C}]$ if and only if one of the following holds: \begin{enumerate}
        \item $\frac{t_{i+1}}{t_i}=t \ \forall \ i$ for some $t>0$,
        \item\label{clause forest weight} $t_0>0,t_1>0$ and $t_i=0 \ \forall \ i>1$,
\item\label{clause trees weight} $t_1>0$ and $t_i=0 \ \forall\  i\not=1$,
\item $t_0>0, t_i=0 \ \forall \ i\geq 0$.  
    \end{enumerate}    \end{enumerate}\end{thm}

Part \ref{main thm clause} of theorem \ref{thm for general set up} is equivalent to the following two propositions. 

Recall the form of the CSG probabilities $\mathbb{P}(C_n)$ as given in \eqref{eqtn0606231}.
\begin{prop}
\label{thm main} Consider a CSG model with couplings $t_0,t_1,t_2,\ldots$, where one or both of $t_0$ and $t_1$ are greater than zero. For all $n\geq 1$, define $a_n:=\sum \mathbb{P}(C_n) C_n$, where the sum is over all posets of cardinality $n$. Then, $\mathbb{R}[a_1,a_2,a_3,\ldots]$ is a subHopf algebra of $\mathbb{R}[\mathcal{C}]$ if and only if one of the following holds:
\begin{enumerate}
\item\label{clause tp} $\frac{t_{k+1}}{t_k}=t \ \forall k$, for some $t>0$ (Transitive Percolation Models),
\item\label{clause forest} $t_0>0,t_1>0$ and $t_k=0 \ \forall k>1$ (Forest Models),
\item\label{clause trees} $t_1>0$ and $t_k=0 \ \forall k\not=1$ (Tree Model),
\item $t_0>0, t_k=0 \forall k> 0$ (Dust Model).\end{enumerate}\end{prop}

\begin{prop}\label{prop converse}
    Consider a CSG model with $t_i=0$ for all $i<N$ and $t_N>0$ for some $N>1$. For all $n\geq N$, define $a_n:=\sum \mathbb{P}(C_n) C_n$, where the sum is over all posets of cardinality $n$. Then, there exists no choice of $a_1,\ldots a_{N-1}$ such that $\mathbb{R}[a_1,a_2,a_3,\ldots]$ is a subHopf algebra of $\mathbb{R}[\mathcal{C}]$.
\end{prop}


The reason for restricting to nonnegative real $t_i$ is because this is standard for CSG models and makes the probabalistic interpretation possible.  Note that Propositions~\ref{thm main} and \ref{prop converse} hold by the same arguments for nonnegative rational $t_i$ (in particular when solving equations in the proof of Proposition~\ref{prop converse} no real but non-rational solutions appear), and so extending scalars, Part \ref{main thm clause} of theorem \ref{thm for general set up} can also be given with $t_i$ nonnegative reals and $\mathbb{K}$ any field of characteristic zero.

 The majority of this section is dedicated to proving a series of lemmas which we will need in order to prove proposition~\ref{thm main}. The proofs of the propositions and the theorem are given in section~\ref{subsec thm proof}. 

The following notation will be useful for us. We will use bold letters to denote unordered lists of positive integers. In particular, given a positive integer $k$ and an integer partition of it into $d$ parts, $k_1,\ldots,k_d$ we will write $\textbf{k}\equiv k_1,\ldots,k_d$. We define the factorial of a list to be equal to the product of the factorials of its entries, $\textbf{k!}\equiv k_1!\cdots k_d! \ .$ We will write $C_{\textbf{k}}$ to denote a poset of cardinality $k$ which has exactly $d$ components of cardinalities $k_1,\cdots,k_d$. When the subscript on a poset is not bold, then it denotes the cardinality of the poset without making any assumptions about its components, \textit{e.g.} $C_l$ is a poset of cardinality $l$.

\begin{definition}\label{definition_gamma}
   Let $a_{k+l}$ denote the sum of cardinality $k+l$ posets weighted by their respective CSG probabilities, as in proposition~\ref{thm main}. Then, for any pair of unlabelled posets, $C_{\normalfont{\textbf{k}}}$ and $C_l$, of cardinalities $k$ and $l$ respectively, we write $\Gamma(C_{\normalfont{\textbf{k}}},C_l)$ to denote the coefficient of $C_{\normalfont{\textbf{k}}}\otimes C_l$ in $\Delta(a_{k+l})$.
\end{definition}
Our notation in definition \ref{definition_gamma} anticipates the asymmetry between the up-sets and the down-sets in the coproduct. In particular, we will see that the $\Gamma$ coefficient will depend on the component decomposition of the up-set but not of the down-set (cf. lemma~\ref{lemma_gamma_form}).


When rebuilding posets out of up-sets and down-sets we will need to keep track of certain features of shuffles of their labels.
\begin{definition}
  Let $w^1, w^2, \ldots, w^{d+1}$ be words, and let $w^1\shuffle w^2 \shuffle \cdots \shuffle w^{d+1}$ denote the set of shuffles of $w^1, w^2, \ldots, w^{d+1}$. 
  \begin{itemize}
      \item For $w\in w^1\shuffle w^2 \shuffle \cdots \shuffle w^{d+1}$, let $v^i_x$ be the number of letters of $w^{d+1}$ which appear before the $x^{th}$ letter of $w^i$ in $w$.
      \item When only the \emph{lengths} of the words matter, we write $Sh(\textbf{k}, l)$ for the set of shuffles of words $w^1, w^2, \ldots, w^d, w^{d+1}$ of lengths $k_1, k_2, \ldots, k_d, l$ respectively.
  \end{itemize}
     
\end{definition}
The only parameters we will need for these shuffles are the $v^i_x$ and so it will never be important to specify the alphabet or any information about the $w^i$ other than their lengths; only the permutation structure of the shuffle will be needed.\footnote{Sometimes shuffles are defined as certain permutations, but we chose the word-based formulation because it keeps the exposition more elementary.}

Recall definition \ref{def multiplicity} of the multiplicity $\mu$, and definition \ref{def template} of templates.
\begin{lemma}\label{lemma_gamma_form} 
Given a set of CSG couplings $t_0,t_1,t_2\ldots$, where one or both of $t_0$ and $t_1$ are greater than zero, and any pair of unlabelled posets, $C_{\emph{\textbf{k}}}$ and $C_l$,
\begin{equation}\begin{split}\label{eqtn lemma_gamma_form}\Gamma(C_{\emph{\textbf{k}}},C_l)=\frac{\mathbb{P}({C}_l)}{\prod_{P\in\mathcal{C}}\mu(P)!\prod_{i=l}^{k+l-1}\lambda(i,0)}  \sum_{Sh({\emph{\textbf{k}}}, l)} \ \prod_{i=1}^d \ \sum_{\lc{C}_{k_i}^i\in temp(C^i_{k_i})}\ \prod_{x\in\lc{C}_{k_i}^i}\lambda^{(v^i_x)}{(\varpi_x,m_x)},\end{split}\end{equation}
where $C_{k_1}^1,C_{k_2}^2,\ldots,C_{k_d}^d$ are the components of $C_\textbf{{\emph{\textbf{k}}}}$; $\mu(P)=\mu^{C_{\emph{\textbf{k}}}}(P)$;
for a fixed template $\lc{C}_{k_i}^i$ and a fixed $x\in \lc{C}^i_{k_i}$, $m_x$ is the number of elements $y\in\lc{C}^i_{k_i}$ such that $y\precdot x$ and $\varpi_x$ is the number of elements $y\in\lc{C}^i_{k_i}$ such that $y\prec x$ and $y\not\precdot x$; and \begin{equation}\label{eqtn1206232}\lambda^{(v^i_x)}(\varpi_x,m_x):=\sum_{r=0}^{v^i_x}\sum_{s=0}^{\varpi_x-m_x}\binom{v^i_x}{r}\binom{\varpi_x-m_x}{s}t_{r+s+m_x}.\end{equation}\end{lemma}

\begin{proof} Recall the notion of labelled cuts from the discussion after Definition~\ref{def coproduct on posets}.
Then we have,
\begin{equation}\label{eq0706231}\Gamma(C_{\textbf{k}},C_l)=\sum_{cuts} \mathbb{P}(\lc{C}_{k+l}),\end{equation} where the sum is over labelled cuts $(\lc{C}_{k+l}, \lc{\mathcal{U}}\otimes\lc{\mathcal{D}})$ with $\lc{\mathcal{U}}$ and $\lc{\mathcal{D}}$ being increasingly labelled representatives of $C_{\textbf{k}}$ and $C_l$ respectively and $\lc{C}_{k+l}$ any naturally labelled poset of cardinality $k+l$.

Each of the labelled cuts contributing to $\Gamma(C_{\textbf{k}},C_l)$ can be constructed as follows:
\begin{itemize}
\item[step 1] Partition the labels: Choose an ordered partition of the interval $[1,k+l]$ into $d+1$ sets $\lc{\mathcal{U}}^1,\lc{\mathcal{U}}^2,\ldots,\lc{\mathcal{U}}^d, \lc{\mathcal{D}}$ of cardinalities $k_1,\ldots,k_d,l$ respectively. The parts of this partition will give the labels for the connected components of $\lc{\mathcal{U}}$ and for $\lc{\mathcal{D}}$ respectively.  Note that this is an ordered partition, so for example when $k=l=1$, $\big(\lc{\mathcal{U}}^1=\{1\}$, $\lc{\mathcal{D}}=\{2\}\big)$ and $\big(\lc{\mathcal{U}}^1=\{2\}$, $\lc{\mathcal{D}}=\{1\}\big)$ constitute distinct choices.
\item[step 2] Choose templates: Upgrade the sets $\lc{\mathcal{U}}^1,\lc{\mathcal{U}}^2,\ldots,\lc{\mathcal{U}}^d, \lc{\mathcal{D}}$ to posets by ordering the elements in $\lc{\mathcal{D}}$ according to some choice of template of $C_l$ and for each $i=1,\ldots,d$ order the elements in $\lc{\mathcal{U}}^i$ according to some choice of template of $C^i_{k_i}$.
\item[step 3] Choose straddling relations: For each $i=1,\ldots,d$, choose a set $\Lambda^i\subseteq \{(x,y): x<y, x\in \lc{\mathcal{D}}, y\in \lc{\mathcal{U}}^i\}$ with the property that for each $y\in \lc{\mathcal{U}}^i$, the $x$ such that $(x,y)\in \Lambda^i$ are incomparable.  The pairs in $\Lambda^i$ give the covering relations $x\precdot y$ between elements of $\lc{\mathcal{D}}$ and $\lc{\mathcal{U}}^i$.  Impose these covering relations for all $i$ in addition to the already chosen poset structures on $\lc{\mathcal{U}}^1,\lc{\mathcal{U}}^2,\ldots,\lc{\mathcal{U}}^d, \lc{\mathcal{D}}$ and take the transitive closure to obtain a poset on $[1,k+l]$ which we call $\lc{C}_{k+l}$.
\end{itemize}
 By construction, $\lc{C}_{k+l}$ is a naturally labelled poset of cardinality $k+l$ and the resulting pair $(\lc{C}_{k+l}, \lc{\mathcal{U}}^1\lc{\mathcal{U}}^2\cdots\lc{\mathcal{U}}^d\otimes\lc{\mathcal{D}})$ is a labelled cut.  Furthermore, every labelled cut can be obtained in this way as given a labelled cut we can read off the $\Lambda^i$, the templates and the partition of the labels.  However, since the partition of the labels was an ordered partition but the decomposition of $C_\textbf{k}$ into connected components does not impose an order on isomorphic connected components, we find that every labelled cut is obtained exactly $\prod_{P\in\mathcal{C}}\mu(P)!$ times.

Therefore,
\begin{equation}\label{eq07062315}\Gamma(C_\textbf{k},C_l)=\frac{1}{\prod_{P\in\mathcal{C}}\mu(P)!}\sum_{step \ 1}\sum_{step\ 2} \sum_{step \ 3}\mathbb{P}(\lc{C}_{k+l}).\end{equation}

\medskip

We can rewrite the sums as follows.
\begin{itemize}
  \item Each choice of step 1 can be viewed as a shuffle of $d+1$ words of lengths $k_1,\ldots,k_d,l$. Denoting these words by $w^1,\ldots,w^d,w^{d+1}$, where $w^i$ has length $k_i$ when $i<d+1$ and length $l$ when $i=d+1$, then the correspondence between a choice of step 1 and a shuffle of words is: the $j^{th}$ smallest integer in $\lc{\mathcal{U}}^i$ ($\lc{\mathcal{D}}$) is the position of the $j^{th}$ letter of $w^i$ ($w^{d+1}$) in the shuffle. Therefore, $\sum_{step \ 1}$ can be rewritten as  $\sum_{Sh(\textbf{k},l)}$.
  \item The sum of step 2 we leave for the moment as a sum over templates, that is $\sum_{step \ 2}$ can be written as $\sum_{temp(C_l)}\sum_{temp(C^1_{k_1})}\cdots\sum_{temp(C^d_{k_d})}$.  The sums over the $temp(C^i_{k_i})$ remain in the statement of the lemma, but we will see below how to simplify the sum over the $temp(C_l)$.
\item The sum of step 3, we leave for the moment as $\sum_{\Lambda^1}\cdots\sum_{\Lambda^d}$.
\end{itemize}

Now consider the summand $\mathbb{P}(\lc{C}_{k+l})$, where $\lc{C}_{k+l}$ corresponds to some fixed choice of shuffle, templates and covering relations $\Lambda^i, i=1,\ldots,d$. For $2\leq x\leq k+l$, let $\lc{D}_{x}$ denote the naturally labelled down-sets of $\lc{C}_{k+l}$, and define the shorthand $\mathbb{P}(x):=\mathbb{P}(\lc{D}_{x-1}\rightarrow\lc{D}_{x})$ and $w(x):=\mathbb{P}(x)\lambda(x-1,0)$. $w(x)$ is the weight (un-normalised probability) of the transition in which $x$ is born. Then,
\begin{align*} \mathbb{P}(\lc{C}_{k+l})=\prod_{x=2}^{k+l}\mathbb{P}(x)& = \left(\frac{\prod_{x\in \lc{\mathcal{D}}}w(x)}{\prod_{x=1}^{l-1}\lambda(x,0)}\right)\left(\frac{ \prod_{i=1}^d\prod_{x\in\lc{\mathcal{U}}^i}w(x)}{\prod_{x=l}^{k+l-1}\lambda(x,0)}\right)\\
  & =\frac{\mathbb{P}(\lc{C}_l)}{\prod_{x=l}^{k+l-1}\lambda(x,0)}\prod_{i=1}^d\prod_{x\in\lc{\mathcal{U}}^i}w(x)=\frac{\mathbb{P}(\lc{C}_l)}{\prod_{x=l}^{k+l-1}\lambda(x,0)}\prod_{i=1}^d w(\Lambda^i),\end{align*} where $\lc{C}_l$ is any naturally labelled representative of $C_l$ since in a CSG model all naturally labelled representatives have the same probability (see Section~\ref{sec cst}), and $w(\Lambda^i):=\prod_{x\in\lc{\mathcal{U}}^i}w(x)$. Therefore given a fixed choice of shuffle and templates, the sum over the straddling relations factorises as $\sum_{step \ 3}\mathbb{P}(\lc{C}_{k+l})=\frac{\mathbb{P}(\lc{C}_l)}{\prod_{x=l}^{k+l-1}\lambda(x,0)}\prod_{i=1}^d\sum_{\Lambda^i}w(\Lambda^i)$.

In a CSG model, we can replace the choice of straddling relations with a choice of proto-past $pr(x)$ for each $x\in\lc{\mathcal{U}}^i$. That is, replace each set of covering relations $\Lambda^i$ with the set $pr(\Lambda^i)$ of proto-past configurations which give rise to the covering relations $\Lambda^i$.\footnote{In analogy with statistical physics, one can think of $\Lambda^i$ and $pr(\Lambda^i)$ as a macroscopic configuration and the set of microscopic configurations which give rise to it, respectively.} Each configuration ${c}\in pr(\Lambda^i)$ is a set of proto-pasts $pr(x)$, one proto-past for each $x\in\lc{\mathcal{U}}^i$. Then, using Remark~\ref{rem CSG by proto-past}, \begin{equation}w(\Lambda^i)=\sum_{{c}\in pr(\Lambda^i)} \prod_{pr(x)\in {c}}t_{|pr(x)|},\end{equation} where $t_{|pr(x)|}$ is the weight of $pr(x)$ and we used the fact that the choices of $pr(x)$ and $pr(y)$ are independent when $y\not=x$.
Hence, \begin{equation}\sum_{\Lambda^i}w(\Lambda^i)=\sum_{\Lambda^i}\sum_{{c}\in pr(\Lambda^i)}\prod_{pr(x)\in {c}}t_{|pr(x)|}=\sum_{\mathscr{C}^i}\prod_{pr(x)\in {c}}t_{|pr(x)|}=\prod_{x\in\lc{\mathcal{U}}^i}\sum_{pr(x)\in\mathscr{C}_x}t_{|pr(x)|},\end{equation} where $\mathscr{C}^i:=\sqcup_{\Lambda^i}pr(\Lambda^i)$ is the set of all proto-past configurations (given fixed shuffle and templates), $\mathscr{C}_x$ is the set of all possible proto-pasts of a single element $x\in\lc{\mathcal{U}}^i$ (given the same fixed shuffle and templates), and the final equality follows from the relation $\mathscr{C}^i=\prod_{x\in\lc{\mathcal{U}}^i}\mathscr{C}_{x}$ where the product symbol denotes the cartesian product.

Now consider some $x\in\lc{\mathcal{U}}^i$ and some proto-past of it $pr(x)\in\mathscr{C}_x$. The weight of $pr(x)$ is $t_{m_x+s+r}$, where $r$ is the number of elements $y\in\lc{\mathcal{D}}$ contained in $pr(x)$, $s$ is the number of elements $y\in\lc{\mathcal{U}}^i$ such that $y\prec x$ but $y\not\precdot x$ which are contained in $pr(x)$, and $m_x$ is the number of elements $y\in\lc{\mathcal{U}}^i$ such that $y\precdot x$. The range for $r$ is $0\leq r\leq \varphi^i_x$, where $\varphi^i_x$ is the number of elements $y\in\lc{\mathcal{D}}$ satisfying $y<x$. The range for $s$ is $0\leq s\leq \varpi_x-m_x$.  Every choice of $r$ and $s$ in their ranges is possible and each choice of subsets of sizes $r$ and $s$ corresponds to a distinct proto-past for $x$ (\textit{i.e.} a distinct element of $\mathscr{C}_x$). Then,
\begin{equation}\label{0906231}\sum_{pr(x)\in\mathscr{C}_x}t_{|pr(x)|}=\sum_{r=0}^{\varphi^i_x}\sum_{s=0}^{\varpi_x-m_x}\binom{\varphi^i_x}{r}\binom{\varpi_x-m_x}{s}t_{r+s+m_x}=\lambda^{(\varphi^i_x)}{(\varpi_x,m_x)},\end{equation}
where the second equality follows from definition \eqref{eqtn1206232}.

Note that given a fixed shuffle, the product $\prod_{x\in\lc{\mathcal{U}}^i}\lambda^{(\varphi^i_x)}{(\varpi_x,m_x)}$ depends on the template $\lc{C}_{k_i}^i$ chosen for ordering the elements of $\lc{\mathcal{U}}^i$.  The product $\prod_{x\in\lc{\mathcal{U}}^i}\lambda^{(\varphi^i_x)}{(\varpi_x,m_x)}$ also depends on the underlying sets of $\lc{\mathcal{U}}^i$ and $\lc{\mathcal{D}}$ in as much as these sets encode the shuffle (\textit{i.e.} through $\varphi^i_x$), but this dependence can be removed by writing the product as,
\begin{equation}\prod_{x\in\lc{\mathcal{U}}^i}\lambda^{(\varphi^i_x)}{(\varpi_x,m_x)}=\prod_{x\in\lc{C}_{k_i}^i}\lambda^{(v^i_x)}{(\varpi_x,m_x)}.\end{equation}

Finally we find, \begin{equation}\begin{split}\sum_{step\ 2} \sum_{step \ 3}\mathbb{P}(\lc{C}_{k+l})&= \sum_{temp(C_n)}\sum_{temp(C^1_{k_1})}\cdots\sum_{temp(C^d_{k_d})}\frac{\mathbb{P}(\lc{C}_l)}{\prod_{x=l}^{k+l-1}\lambda(x,0)}\prod_{i=1}^d\prod_{x\in\lc{C}_{k_i}^i}\lambda^{(v^i_x)}{(\varpi_x,m_x)}\\
&=\frac{\mathbb{P}(\lc{C}_l)}{\prod_{x=l}^{k+l-1}\lambda(x,0)}\bigg(\sum_{temp(C_l)}1\bigg)\prod_{i=1}^d\ \sum_{\lc{C}_{k_i}^i\in temp(C^i_{k_i})}\ \prod_{x\in\lc{C}_{k_i}^i}\lambda^{(v^i_x)}{(\varpi_x,m_x)}\\
    &=\frac{\mathbb{P}(C_l)}{\prod_{x=l}^{k+l-1}\lambda(x,0)}\prod_{i=1}^d\ \sum_{\lc{C}_{k_i}^i\in temp(C^i_{k_i})}\ \prod_{x\in\lc{C}_{k_i}^i}\lambda^{(v^i_x)}{(\varpi_x,m_x)},\end{split}\end{equation}
where in the last line we used equation \eqref{eqtn0606231} for the CSG probability of an unlabelled poset.
The result follows.
\end{proof}


\begin{cor}\label{cor nice coprod}
    For all $n>0$, let $a_n$ be defined via the CSG probabilities as in proposition~\ref{thm main}. If $\mathbb{R}[a_1,a_2,a_3,\ldots]$ is a subHopf algebra of $\mathbb{R}[\mathcal{C}]$, then $\Delta(a_n)=\sum_{k=0}^n P_{k,n}(a_1,\ldots, a_k)\otimes a_{n-k}$, where $P_{k,n}(a_1,\cdots, a_{k})$ is a polynomial homogeneous of degree $k$ and $a_0=1$.
\end{cor}
\begin{proof}
         By definition \ref{definition_gamma} of $\Gamma(C_\textbf{k},C_l)$, we may write the coproduct of the generator $a_n$ as,
                   \begin{equation}\label{eq3105241}
    \begin{split}
    \Delta(a_n)=&\sum_{k=0}^n \sum_{\normalfont\textbf{k}\in p(k)} \sum_{C_\textbf{k}}\sum_{C_{n-k}}\Gamma(C_\textbf{k},C_{n-k}) \ C_{\textbf{k}}\otimes C_{n-k},  \\
    \end{split}
    \end{equation}
where $p(k)$ is the set of integer partitions of $k$, and the sums over $C_{\textbf{k}}$ and $C_{n-k}$ run over all posets with the corresponding cardinalities (see the discussion of notation before definition \ref{definition_gamma}).

    An immediate corollary of lemma \ref{lemma_gamma_form} is that $\Gamma(C_\textbf{k},C_l)$ has a separable form so that it can be written as $\Gamma(C_{\textbf{k}},C_l)=\tilde{\Gamma}(C_{\textbf{k}})\mathbb{P}(C_l)$, for some function $\tilde{\Gamma}(C_{\textbf{k}})$. Substituting this separable form into \eqref{eq3105241} yields,
    \begin{equation}
    \begin{split}
    \Delta(a_n)=&\sum_{k=0}^n  \underbrace{\bigg(\sum_{\normalfont\textbf{k}\in p(k)}\sum_{C_{\textbf{k}}}\tilde{\Gamma}(C_{\textbf{k}}) \ C_{\textbf{k}}\bigg)}_{Q_{k,n}}\otimes a_{n-k},\\
    \end{split}
    \end{equation}
where we note that $a_{n-k}\not=0 \ \forall \ k$ by the statement of corollary \ref{cor nice coprod}.

As $a_k$ is homogeneous of degree $k$ for any $k$, the $a_k$ are linearly independent.
Therefore, there exists a family $\delta_k$ of linear forms on $\mathbb{R}[\mathcal{C}]$ such that,
\begin{equation}\label{linear forms}
    \delta_k(a_l) = \begin{cases}
      1 &  \text{ if }l=k\\
      0  & \text{ otherwise.}\\
    \end{cases} 
\end{equation} 
Additionally, since $\mathbb{R}[a_1,a_2,\ldots ]$ is a subHopf algebra of  $\mathbb{R}[\mathcal{C}]$, we have,
\begin{equation}
    (id\otimes \delta_{n-k})\Delta(a_n)=Q_{k,n} \in\mathbb{R}[\mathcal{C}] \ \forall \ k.
\end{equation}
Hence, there exists a polynomial $P_{k,n}(a_1,\ldots, a_n)$ such that, $Q_{k,n}=    P_{k,n}(a_1, a_2, \ldots , a_n)$. As $\Delta$ is
homogeneous, $P_{k,n}$ is homogeneous of degree $k$, so in particular does not depend on $a_l$ if
$l > k$. This completes the proof.
\end{proof}

It will be convenient to have notation for the coefficients of the polynomials of the previous corollary.  The next corollary sets this notation.
\begin{cor}
      For all $n>0$, let $a_n$ be defined via the CSG probabilities as in proposition~\ref{thm main}. If $\mathbb{R}[a_1,a_2,a_3,\ldots]$ is a subHopf algebra of $\mathbb{R}[\mathcal{C}]$, then the coproduct of its generators can be written as,
\begin{equation}\begin{split}\label{coproduct coefficients}
&\Delta(a_n)=\sum_{k+l=n}\sum_{d=1}^k\sum_{\normalfont\textbf{k}\in p_d(k)}\beta_{\normalfont\textbf{k},l} \ a_{k_1}\ldots a_{k_d} \otimes a_l,\\
\end{split}\end{equation}
where $p_d(k)$ is the set of integer partitions $\normalfont\textbf{k}\equiv k_1,\ldots,k_d$ of $k$ into $d$ parts. 
\end{cor}


\subsection{Transitive Percolation}

\begin{definition}
    For $n$ and $k$ nonnegative integers, the $q$-binomial coefficient is
    \[
        \begin{bmatrix} n \\ k \end{bmatrix}_q = \frac{(1-q^n)(1-q^{n-1})\cdots (1-q^{n-k+1})}{(1-q)(1-q^2)\cdots (1-q^k)}.
    \]
\end{definition}
The $q$-binomial coefficients have many nice properties generalizing properties of the usual binomial coefficients.  The facts that we will make use of are the observation that taking the limit as $q\rightarrow 1$ gives the usual binomial coefficients along with an identity from \cite{Butler:2010}.

\begin{lemma}\label{tp lemma} For all $n>0$ let $a_n$ be defined via the Transitive Percolation models, as in proposition~\ref{thm main}, clause~\ref{clause tp}. Then $\mathbb{R}[a_1,a_2,\ldots]$ is a co-commutative Hopf algebra with coproduct coefficients given by $\beta_{\normalfont\textbf{k},l}=0$ when $d>1$ and $\beta_{k,l}=\beta_{l,k}= \begin{bmatrix}
l+k \\
l 
\end{bmatrix}_{q},$
where $q$ is the Transitive Percolation parameter defined in \eqref{pqdefinition}, when $d=1$, $\normalfont\textbf{k}=k_1=k$. 
\end{lemma}
\begin{proof} Using the Transitive Percolation parameters \eqref{pqdefinition} we evaluate the product,
\begin{equation}\begin{split} \prod_{x\in\lc{C}_{k_i}^i}\lambda^{(v^i_x)}{(\varpi_x,m_x)}=q^{-\sum_{x=1}^{k_i} v_x^i}q^{-R^i}p^{L^i},\end{split}\end{equation}
where $L^i$ and $R^i$ are the total number of links and relations in the component $\lc{C}^i_{k_i}$. This expression is independent of the templates so we can perform the sums in \eqref{eqtn lemma_gamma_form} to get,
\begin{equation}\label{tp_gamma_initial}\begin{split}\Gamma(C_{\textbf{k}},C_l)=\frac{\mathbb{P}({C}_l) q^{-R}p^{L}\prod_{i=1}^d \ \Psi({C}_{k_i}^i)}{\prod_{P\in\mathcal{C}}\mu(P)!\prod_{x=l}^{k+l-1}q^{-x}} \sum_{Sh(\textbf{k},l)}\ \ q^{-\sum_{i=1}^d\sum_{x=1}^{k_i} v_x^i},\end{split}\end{equation}
where $R=\sum_iR^i$ and $L=\sum_iL^i$ are the total number of relations and links in $C_{\textbf{k}}$, and where we used the fact that in transitive percolation $\lambda(x,0)=q^{-x}$.

 Using relations \eqref{component_template_number} and \eqref{TP_probability}, we manipulate the RHS to get,
\begin{equation}\label{eqtn tp proof}\begin{split}\Gamma(C_{\textbf{k}},C_l)=\mathbb{P}({C}_l)\mathbb{P}({C}_{\textbf{k}})\frac{q^{kl} \ \textbf{k!}}{k!} \sum_{Sh(\textbf{k},l)} \ \ q^{-\sum_{i=1}^d\sum_{x=1}^{k_i} v_x^i}.\end{split}\end{equation}
Now consider the sum $ \sum_{Sh(\textbf{k},l)} \ \ q^{-\sum_{x,i} v_x^i}$. Because the summand treats the words of length $k_1,\ldots,k_d$ and the letters they contain on an equal footing, we can rewrite the sum as a shuffle over two words of lengths $k$ and $l$ respectively, \begin{equation}\sum_{Sh(\textbf{k},l)} \ \ q^{-\sum_{x,i} v_x^i}=\frac{k!}{\textbf{k!}} \sum_{Sh(k,l)} \ \ q^{-\sum_{x=1}^k v^1_x},\end{equation} where $\frac{k!}{\textbf{k!}}=\binom{k}{k_1,\ldots, k_d}$ is the number of shuffles of $d$ words of length $k_1,\ldots,k_d$ into a single word of length $k$. Plugging this back into \eqref{eqtn tp proof} we have,
\begin{equation}\begin{split}\Gamma(C_\textbf{k},C_l)=\mathbb{P}({C}_l)\mathbb{P}(C_\textbf{k})\begin{bmatrix}
l+k \\
l 
\end{bmatrix}_{q},\end{split}\end{equation}
where we manipulated the sum as,
\begin{equation}
\begin{split}
    q^{kl}\sum_{Sh(k,l)} \ \ q^{-\sum_{x=1}^k v^1_x}=q^{kl}\sum_{v_1^1\leq v_2^1 \cdots \leq v_k^1\leq l}q^{-v_1^1-v_2^1\cdots -v_k^1}=\ \begin{bmatrix}
l+k \\
l 
\end{bmatrix}_{q}, \text{ by \cite{Butler:2010}}.
\end{split}
\end{equation}
 
The result follows by comparison of definition \ref{definition_gamma} of $\Gamma$ with the definition of $\beta$ in \eqref{coproduct coefficients}.
\end{proof}

As a consistency check we note that the $\beta$ coefficients of dust model, in which $a_n$ is the $n$-antichain, is recovered in the limit $q\rightarrow 1$ where we have $\beta_{k,l}=\binom{k+l}{l}$.

\subsection{The Forest Models, Tree Model and Connes-Moscovici}

In this section we prove that the CSG Forest models are Hopf. More precisely, we prove that the CSG Forest generators $a_n$ generate the Connes-Moscovici Hopf algebra (lemma \ref{lem forest hopf by iso}), thus providing a new collection of combinatorially-meaningful generating sets. We give a recursive formula for their coproduct coefficients, $\beta_{\textbf{k},l}$ (lemma \ref{lemma forest beta}). This formula is combinatorial in nature, since it involves the enumeration of shuffles and the enumeration of forest partitions, the trees they contain and the components of those trees. From these $\beta_{\textbf{k},l}$, the coproduct coefficients of closely related Hopf algebras of trees can be derived, and we give these explicitly in table \ref{table beta forests}. For the case of $\textbf{k}=k_1=k$, we give a closed form algebraic expression for $\beta_{\textbf{k},l}$ 
 as a weighted sum of binomial coefficients (table \ref{table beta to first order forests}). As an example, we use our formulae to compute the coproduct of the generators of degree 2, 3 and 4 in an un-normalised variation of the forest models. Setting $t_0=0, t_1=1$ in these expressions yields the coproduct of the usual Connes-Moscovici generators $\delta_n$. The latter were previously computed in \cite{MENOUS200575}, providing a consistency check for our results.

\begin{lemma}\label{lem forest hopf by iso} For all $n>0$, let $a_n$ be defined via the CSG Forest models as in proposition \ref{thm main}, clause \ref{clause forest}, and let $\delta_n$ denote the Connes-Moscovici generators as defined in \eqref{CM generators}. Then, 
$\mathbb{R}[a_1,a_2,\ldots]=\mathbb{R}[\delta_1,\delta_2,\ldots]$.
\end{lemma}

\begin{proof}
Scaling the generators $a_n$ by non-zero elements of $\mathbb{R}$ does not change the algebra $\mathbb{R}[a_1,a_2,\ldots]$. Therefore, for convenience we scale $a_n$ by a factor of $t_0t_1^{-n}\prod_{x=1}^{n-1}(t_0+xt_1)$. The rescaled generator is the sum of forests $f$ with $n$ vertices, each weighted $\Psi(f)t^{\tau}$, where $\Psi(f)$ is the number of natural labellings of $f$, $\tau$ is the number of trees in $f$ and $t=\frac{t_0}{t_1}$. We write $a_n(t)$ to denote this scaled generator.  

Consider $a_n(1)$, it is the sum of forests with $n$ vertices each weighted by their number of natural labellings.  Note that $\delta_n = B_+(a_{n-1}(1))$ for $n\geq 2$, where $\delta_n$ are the homogeneous generators for the Connes-Moscovici Hopf algebra as given in \eqref{CM generators} and $B_+$ is, as before, the add-a-root operator, however in this context we're working with trees with roots as minimal elements, so $B_+$ adds a new minimal element below all other elements. To see this, observe that taking $B_+$ of a forest does not change the number of natural labellings as the root must always take the smallest label, and we already know that the generators of Connes-Moscovici are precisely the rooted trees of each size weighted by their number of natural labellings.  

We also know that Connes-Moscovici is Hopf and in particular its coproduct has the following form
\[
\Delta(\delta_n) = \sum_{j=0}^n P_{j,n}(\delta_1, \ldots, \delta_j)\otimes \delta_{n-j}
\]
where $P_{j,n}$ is a polynomial homogeneous of degree $j$.  The exact form of the $P_{j,n}$ will not be important, except for the following two observations.  Since $\delta_{n+1} = B_+(a_n(1))$, the only term of the form $x\otimes \bullet$ in $\Delta(\delta_{n+1})$ is $a_n(1)\otimes \bullet$ and so,
\begin{equation}\label{eq b to a}
    a_n(1) = P_{n,n+1}(\delta_1, \ldots, \delta_n).
\end{equation}
Furthermore, the trees in $a_n(1)$ are the same trees as in $\delta_n$ with the same weights (both weighted by the number of natural labellings), so, 
\[
P_{n, n+1} (\delta_1, \ldots, \delta_n) = \delta_n + (\text{products of $\delta_k$, $k<n$}).
\]

Moving from $a_n(1)$ to $a_n(t)$ the only difference is the power of $t$ counting the number of trees in each forest.  Since each $\delta_j$ consists of sums of single trees, we get from \eqref{eq b to a},
\begin{equation}\label{eq b to at}
a_n(t) = P_{n, n+1}(t\delta_1, \ldots, t\delta_n),
\end{equation}
so,
\[
a_n(t) = t\delta_n + (\text{products of $t\delta_k$, $k<n$}).
\]
Therefore, using $t\delta_1=a_1(t) = t\bullet$, inductively for any $t>0$ we can invert the system of equations given by \eqref{eq b to at} to obtain $\delta_n$ as a polynomial in $a_1(t), \ldots, a_n(t)$. This is the desired automorphism between the Forest Models and the Connes-Moscovici Hopf algebra giving the map between their generators explicitly.
\end{proof}

\begin{remark}
    The argument in the proof of Lemma~\ref{lem forest hopf by iso} does not directly give the form of the coproduct for the original generators $a_n$ that we know holds for the Forest models by Corollary~\ref{cor nice coprod}.  In the particular case of $t=1$ we can obtain this form algebraically by continuing the proof of Lemma~\ref{lem forest hopf by iso} as follows. 

     By the 1-cocycle property of $B_+$ (see \cite{ck0}) we have,
    \[
    \Delta(\delta_{n+1}) = \Delta(B_+(a_{n}(1))) = (\text{id}\otimes B_+)\Delta(a_n(1)) + B_+(a_n(1))\otimes 1,
    \]
    so, writing $B_-$ for the operation of removing the root from a tree we get,
    \begin{align*}
        \Delta(a_n(1)) & = (\text{id}\otimes B_-) (\Delta(\delta_{n+1}) - \delta_{n+1}\otimes 1) \\
        & = (\text{id}\otimes B_-)\left(\sum_{j=0}^{n} P_{j, n+1}(\delta_0, \ldots, \delta_j) \otimes \delta_{n+1-j} \right)\\
        & = \sum_{j=0}^{n} P_{j, n+1}(\delta_0, \ldots, \delta_j) \otimes a_{n-j}(1).
    \end{align*}
    Subbing in for each $\delta_k$ on the left hand side of the tensor products its expression as a polynomial in $a_0(1), \ldots, a_k(1)$ from Lemma~\ref{lem forest hopf by iso} we obtain an expression for $\Delta(a_n(1))$ of the form
    \[
    \Delta(a_n(1)) = \sum_{j=0}^{n} P'_{j, n}(a_0(1), \ldots, a_j(1)) \otimes a_{n-j}(1) 
    \]
    where the $P'$ are some different polynomials with once again the property that $P'_{j,n}$ is homogeneous of degree $j$ (with $a_k(1)$ taken to have degree $k$), giving an alternate proof of Corollary~\ref{cor nice coprod} in this particular case.
\end{remark}

In what follows, we give a combinatorial formulation of the coproducts of the $a_n$ and the $\delta_n$.
Recall the definition of the coproduct coefficients, $\beta_{\textbf{k},l}$, as given in \eqref{coproduct coefficients}, and definition \ref{def forest partition} of forest partitions.
\begin{lemma}\label{lemma forest beta} In the Forest Models, as defined in proposition \ref{thm main}, clause \ref{clause forest}, the coproduct coefficients are given by, 
\begin{equation}\label{forest beta}\begin{split}
\beta_{\emph{\textbf{k}},l}=\frac{\prod_{k_i\in\emph{\textbf{k}}}\prod_{x=1}^{k_i-1}(t_0+xt_1)}{\prod_{\gamma\in\mathbb{Z}_{\geq 1}}\mu^{\emph{\textbf{k}}}(\gamma)!}\bigg(& \frac{1}{\prod_{x=l}^{k+l-1}(t_0+xt_1)}\sum_{Sh(\emph{\textbf{k}},l)} \prod_{i=1}^d (t_0+v^i_1t_1)-\frac{t_0^d}{\emph{\textbf{k!}}}B_{\normalfont\textbf{k},l}\bigg),\end{split}
\end{equation}
where $B_{\normalfont\textbf{k},l}$ can be evaluated using any choice of forest $C_{\normalfont\textbf{k}}$ via,
\begin{equation}\label{forest B}
B_{\normalfont\textbf{k},l}:=  \sum_{\emph{\emph{\normalfont\textbf{n}}}\not=\emph{\textbf{k}}}{\beta}_{\normalfont \textbf{n},l}\ \bigg(\frac{\normalfont\textbf{n!}\prod_{\gamma\in \mathbb{Z}_{\geq 1}}  \mu^{\normalfont\textbf{n}}(\gamma)!}{\prod_{n_i\in \normalfont\textbf{n}}\prod_{x=0}^{n_i-1}(t_0+xt_1)} \bigg)\sum_{\substack{\pi \text{ s.t.} \\ N(\pi)=\emph{\textbf{n}}}}\frac{1}{\prod_{P\in \mathcal{P}}\mu^{\pi}(P)!}\prod_{P\in\mathcal{C}}\frac{\mu(P)!}{\prod_{F\in\pi}\mu^F(P)!},
\end{equation}
where the first sum is over unordered lists $\normalfont\textbf{n}$ of positive integers, the second sum is over forest partitions $\pi$ of $C_{\normalfont\textbf{k}}$, and $\mu(P)=~\mu^{C_{\normalfont\textbf{k}}}(P)$.
\end{lemma}

\begin{proof}
We will obtain an expression for $\beta_{\textbf{k},l}$ by equating two expressions for $\Gamma(C_{\textbf{k}},C_{l})$, where $C_{\textbf{k}}$ and $C_{l}$ are forests with roots as minimal elements of the poset. (In the Forest models, $\Gamma(C_{\textbf{k}},C_{l})=0$ if $C_{\textbf{k}}$ or $C_{l}$ are not forests.) We get our first expression by plugging the Forest models coefficients, $t_k=0$ for all $k>1$, into \eqref{eqtn lemma_gamma_form},
\begin{equation}\begin{split}\label{Forest Gamma 1} \Gamma(C_\textbf{k},C_l)=&\frac{\mathbb{P}({C}_l) \ t_1^{k-d}\ \prod_{i=1}^d \ \Psi({C}_{k_i}^i)}{\prod_{P\in\mathcal{C}}\mu(P)!\prod_{x=l}^{k+l-1}(t_0+xt_1)} \sum_{Sh(\textbf{k},l)} \prod_{i=1}^d (t_0+v^i_1t_1)\\
=&\frac{\mathbb{P}({C}_l) \mathbb{P}({C}_k)}{k!}\prod_{x=0}^{k-1}(t_0+xt_1) \frac{\textbf{k}!\ t_0^{-d}}{\prod_{x=l}^{k+l-1}(t_0+xt_1)}\sum_{Sh(\textbf{k},l)} \prod_{i=1}^d (t_0+v^i_1t_1),\end{split}\end{equation}
where in the first line we were able to perform the sum over templates since the product $ \prod_{x\in\lc{C}_{k_i}^i}\lambda^{(v^i_x)}{(\varpi_x,m_x)}=(t_0+v^i_1t_1)t_1^{k_i-1}$ is label independent, and in the second line we manipulated our expression using \eqref{component_template_number} and \eqref{forest_probability}.

We obtain a second expression for $\Gamma(C_{\textbf{k}},C_{l})$ by comparing definition \ref{definition_gamma} of $\Gamma$ with the definition of $\beta$ in \eqref{coproduct coefficients} to obtain,
\begin{equation}\label{Forest Gamma 2} \begin{split}
\Gamma(C_{\textbf{k}},C_l)=&\mathbb{P}(C_l)\sum_{\pi}G_{\pi}\beta_{N(\pi),l}\prod_{F\in \pi} \mathbb{P}(F)\\
=&\frac{\mathbb{P}(C_l)\mathbb{P}(C_{\textbf{k}})}{k!}\prod_{x=0}^{k-1}(t_0+xt_1)\sum_{\pi}G_{\pi}{\beta}_{N(\pi),l}\prod_{i\in N(\pi)}\frac{i!}{\prod_{x=0}^{i-1}(t_0+xt_1)} \prod_{P\in\mathcal{C}}\frac{\mu(P)!}{\prod_{F\in\pi}\mu^F(P)!},
\end{split}\end{equation}
where $\pi$ denotes a partition of $C_\textbf{k}$ and $N(\pi)$ denotes the list of cardinalities of the forests in $\pi$ (see definition~\ref{def forest partition}); $\mu(P)=\mu^{C_\textbf{k}}(P)$;
$G_{\pi}$ is the product of multinomial coefficients which we get from the expansion of the product of generators in \eqref{coproduct coefficients} and is given by,
\begin{equation}\label{eqtn G pi}G_{\pi}=\frac{\prod_{\gamma\in\mathbb{Z}_{\geq 1}}\mu^{N(\pi)}(\gamma)!}{\prod_{P\in \mathcal{P}}\mu^{\pi}(P)!},\end{equation}
and in the second line we manipulated our expression using \eqref{nooftemplatesforest} and \eqref{forest_probability}.


Note that $\beta_{\textbf{k},l}$ appears in exactly one term in \eqref{Forest Gamma 2} -- in the term corresponding to the unique partition of $C_{\textbf{k}}$ into its components. The result follows by equating \eqref{Forest Gamma 1}
and \eqref{Forest Gamma 2} and rearranging for $\beta_{\textbf{k},l}$. In particular, the expression for $B_{\textbf{k},l}$ is obtained by substituting expression \eqref{eqtn G pi} for $G_{\pi}$ and by re-ordering the terms in the sum over partitions by grouping together all partitions $\pi$ which share the same $N(\pi)$.

We now comment on the validity of our result. 

Firstly, note that while a tree $C_{\textbf{k}}$ must be chosen in order to compute $B_{\textbf{k},l}$, it is a corollary of the result that the Forest Models are Hopf (cf. lemma~\ref{lem forest hopf by iso}) and of the definition of the $\beta$ coefficients in these models (cf. \eqref{coproduct coefficients}) that the value of $B_{\textbf{k},l}$ will be independent of this choice.

Secondly, for each list $\textbf{n}$, the value of the sum over $\pi \ s.t. \ N(\pi)=\textbf{n}$ is non-vanishing only if $\textbf{n}$ can be obtained from $\textbf{k}$ by combining some of its entries. For example, when $\textbf{k}=k_1,k_2$ the only contribution to $B_{\textbf{k},l}$ comes from the single-entry list $\textbf{n}=k_1+k_2=k$. When $\textbf{k}=k_1,k_2,k_3$, the contributions to $B_{\textbf{k},l}$ come $\textbf{n}=k,\ \textbf{n}=k_1+k_2,k_3,\ \textbf{n}=k_1+k_3,k_2$ and $\textbf{n}=k_2+k_3,k_1$, although depending on the symmetries of $\textbf{k}$ some of these contributions may be equal and should not be over-counted, \textit{e.g.} when $k_1=k_2=k_3$ there are only two terms in the sum corresponding to $\textbf{n}=k$ and $\textbf{n}=k_1+k_2,k_3=k_1+k_3,k_2=k_3+k_2,k_1$. The upshot is that the $\beta_{\textbf{k},l}$ are well-defined recursively: each $\beta_{\textbf{k},l}$ depends only on a finite number of $\beta_{\textbf{n},l}$ where the length of $\textbf{n}$ is strictly smaller than the length of $\textbf{k}$. In particular, $B_{\textbf{k},l}=0$ when $\textbf{k}=k_1=k$.  
\end{proof}

\begin{table}[htbp!]
\begin{tabular}{ |m{\textwidth}|} 
   \hline
   \textbf{Normalised forests}
\hfill   (proposition \ref{thm main}, clause \ref{clause forest}) \\
Define  $t=\frac{t_1}{t_0}.$  
   \begin{align} \label{normalised forest beta version 2}  \beta_{{\textbf{k}},l}=\frac{\prod_{k_i\in{\textbf{k}}}\prod_{x=1}^{k_i-1}(1+xt)}{\prod_{\gamma\in\mathbb{Z}_{\geq 1}}\mu^{{\textbf{k}}}(\gamma)!}\bigg(\frac{1}{\prod_{x=l}^{k+l-1}(1+xt)}\sum_{Sh({\textbf{k}},l)} \prod_{i=1}^d (1+v^i_1t)-\frac{1}{{\textbf{k!}}}B_{\normalfont\textbf{k},l}\bigg)
  \end{align}
\begin{align}\label{normalised forest B version 2}   B_{\normalfont\textbf{k},l}=  \sum_{{{\normalfont\textbf{n}}}\not={\textbf{k}}}{\beta}_{\normalfont \textbf{n},l}\ \bigg(\frac{\normalfont\textbf{n!}\prod_{\gamma\in \mathbb{Z}_{\geq 1}}  \mu^{\normalfont\textbf{n}}(\gamma)!}{\prod_{n_i\in \normalfont\textbf{n}}\prod_{x=0}^{n_i-1}(1+xt)} \bigg)\sum_{\substack{\pi \text{ s.t.} \\ N(\pi)=\emph{\textbf{n}}}}\frac{1}{\prod_{P\in \mathcal{P}}\mu^{\pi}(P)!}\prod_{P\in\mathcal{C}}\frac{\mu(P)!}{\prod_{F\in\pi}\mu^F(P)!}
\end{align}
\\

  \hline
  \textbf{Un-normalised forests}
\hfill (theorem \ref{thm for general set up}, clause \ref{clause forest weight})  

  \begin{flalign}\label{unnormalised forest beta}
      \beta_{{\textbf{k}},l}=\frac{1}{\prod_{\gamma\in\mathbb{Z}_{\geq 1}}\mu^{{\textbf{k}}}(\gamma)!}\bigg(\sum_{Sh({\textbf{k}},l)} \prod_{i=1}^d (t_0+v^i_1t_1)-\frac{t_0^d}{{\textbf{k!}}}B_{\normalfont\textbf{k},l}\bigg)
  \end{flalign}
\begin{flalign}\label{unnormalised forest B}
B_{\normalfont\textbf{k},l}= \sum_{{{\normalfont\textbf{n}}}\not=\normalfont{\textbf{k}}}{\beta}_{\normalfont \textbf{n},l}\ t_0^{-|\normalfont\textbf{n}|} \ \normalfont\textbf{n!}\prod_{\gamma\in \mathbb{Z}_{\geq 1}}  \mu^{\normalfont\textbf{n}}(\gamma)!\sum_{\substack{\pi \text{ s.t.} \\ N(\pi)=\emph{\textbf{n}}}}\frac{1}{\prod_{P\in \mathcal{P}}\mu^{\pi}(P)!}\prod_{P\in\mathcal{C}}\frac{\mu(P)!}{\prod_{F\in\pi}\mu^F(P)!}    
\end{flalign}
\\
  \hline
  \textbf{Normalised trees}
\hfill  (proposition \ref{thm main}, clause \ref{clause trees})  

    \begin{equation}\begin{split}
\beta_{\normalfont\textbf{k},l}= \frac{\Gamma(l)\prod_{k_i\in\normalfont\normalfont\textbf{k}}\Gamma(k_i)}{\Gamma(k+l)}\frac{\sum_{Sh(\normalfont\textbf{k},l)} \prod_{i=1}^{d} v^i_1}{\prod_{\gamma\in\mathbb{Z}_{\geq 1}}\mu^{\normalfont\textbf{k}}({\gamma})!}\end{split}
\end{equation}\\
  \hline    
    \textbf{Un-normalised trees}
\hfill      (theorem \ref{thm for general set up}, clause \ref{clause trees weight})  

        \begin{equation}\begin{split}
\beta_{\normalfont\textbf{k},l}=\frac{\sum_{Sh(\normalfont\textbf{k},l)} \prod_{i=1}^{d} v^i_1t_1}{\prod_{\gamma\in\mathbb{Z}_{\geq 1}}\mu^{\normalfont\textbf{k}}({\gamma})!}\end{split}\end{equation} \\
\hline
 \textbf{Connes-Moscovici}
 \hfill   (theorem \ref{thm for general set up}, clause \ref{clause trees weight} with $t_1=1$)
    
 \begin{align}
  \beta_{\normalfont\textbf{k},l}=\frac{\sum_{Sh(\normalfont\textbf{k},l)} \prod_{i=1}^{d} v^i_1}{\prod_{\gamma\in\mathbb{Z}_{\geq 1}}\mu^{\normalfont\textbf{k}}({\gamma})!}   
 \end{align}\\
  \hline
  \end{tabular}
   \caption{The coproduct coefficients in the various forest and tree models. \label{table beta forests} Expressions \eqref{normalised forest beta version 2} and \eqref{normalised forest B version 2} are equal to expressions 
  \eqref{forest beta} and \eqref{forest B} , respectively - the only difference being that the former are manifestly independent of the absolute values of $t_1$ to $t_0$. The tree expressions are obtained from the forest expressions by setting $t_0=0$. The un-normalised expressions are obtained using an un-normalised version of equation \eqref{eqtn lemma_gamma_form} in which the probabilties are replaced by weights.}
\end{table}

\begin{table}[ht]
\begin{tabular}{ |m{\textwidth}|} 
   \hline
   \textbf{Normalised forests}
   \hfill
   (proposition \ref{thm main}, clause \ref{clause forest})  
   \begin{flalign}   \beta_{k,l}=\frac{\prod_{x=1}^{k-1}(1+xt)}{\prod_{x=l}^{l+k-1}(1+xt)}\sum_{i=1}^{l+1} \binom{l+k-i}{l+1-i}(1+(i-1)t)
  \end{flalign}
\\
  \hline
  \textbf{Un-normalised forests}
\hfill (theorem \ref{thm for general set up}, clause \ref{clause forest weight})  

  \begin{flalign}
  \beta_{k,l}=    \sum_{i=1}^{l+1} \binom{l+k-i}{l+1-i}(t_0+(i-1)t_1)
  \end{flalign}

\\
  \hline
  \textbf{Normalised trees}
\hfill
  (proposition \ref{thm main}, clause \ref{clause trees})  

    \begin{equation}\begin{split}
\beta_{k,l}=\frac{(k-1)!(l-1)!}{(k+l-1)!}\sum_{i=1}^{l+1} \binom{l+k-i}{l+1-i}(i-1)\end{split}
\end{equation}\\
  \hline    
    \textbf{Un-normalised trees}
\hfill      (theorem \ref{thm for general set up}, clause \ref{clause trees weight})  

        \begin{align}
\beta_{k,l}=\sum_{i=1}^{l+1} \binom{l+k-i}{l+1-i}(i-1)t_1  
 \end{align}\\
\hline
 \textbf{Connes-Moscovici}
\hfill   (theorem \ref{thm for general set up}, clause \ref{clause trees weight} with $t_1=1$)
    
 \begin{align}
\beta_{k,l}=\sum_{i=1}^{l+1} \binom{l+k-i}{l+1-i}(i-1)  
 \end{align}\\
  \hline
  \end{tabular}
   \caption{The coproduct coefficients in the special case when $\textbf{k}=k_1=k$, that is when there is only one factor on the left hand side of the coproduct, in the various forest and tree models. \label{table beta to first order forests}}
\end{table}

In table \ref{table beta forests}, we give the corresponding expressions for the coproduct coefficients in the various forest and tree algebras which are referred to in theorem \ref{thm for general set up} and proposition \ref{thm main}.
In table \ref{table beta to first order forests}, we give the algebraic expressions for the coproduct coefficients in the special case when $\textbf{k}=k_1=k$. In table \ref{table more betas}, we explicitly compute  several coproduct coefficients in the un-normalised forest model (theorem \ref{thm for general set up}, clause \ref{clause trees weight}) with $\textbf{k}\not=k$ .  Putting together the results from all three tables, we find that in the un-normalised forest models we have the following,
  \begin{equation}\label{coproduct example computation}
  \begin{split}
   &\tilde\Delta(a_2)=(t_0+2t_1)\ a_1\otimes a_1,\\
   &\tilde\Delta(a_3)= \big((3t_0+t_1)\ a_2+t_1(t_1+2t_0) \ a_1a_1\big)\otimes a_1+(3t_0+3t_1) \ a_1\otimes a_2,\\
   &\tilde\Delta(a_4)=\big((4t_0+t_1)\ a_3+(7t_1t_0+3t_1^2)\ a_2a_1+(t_1^3-2t_0^2t_1+t_0t_1^2)\ a_1a_1a_1\big)\otimes a_1\\
   &\hspace{10mm}+\big ((6t_0+4t_1)\ a_2+ (7t_1^2+8t_1t_0)\ a_1 a_1\big)\otimes a_2+(4t_0+6t_1)\ a_1\otimes a_3.
        \end{split}
  \end{equation}
  Setting $t_0=0, t_1=1$ in \eqref{coproduct example computation} yields the coproduct of the corresponding $\delta_n$ which were previously computed in \cite{MENOUS200575}.

  Finally, we note that defining, \begin{equation}\label{fn eqtn}\begin{split}
      f_{\textbf{n}}(C_\textbf{k})&:=\sum_{\substack{\pi \text{ s.t.} \\ N(\pi)=\textbf{n}}}\frac{1}{\prod_{P\in \mathcal{P}}\mu^{\pi}(P)!}\prod_{P\in\mathcal{C}}\frac{\mu(P)!}{\prod_{F\in\pi}\mu^F(P)!},\\
    \end{split}
\end{equation}
where the sum is over paritions $\pi$ of $C_{\textbf{k}}$ and $\mu(P)=\mu^{C_{\textbf{k}}}(P)$, we can express \eqref{unnormalised forest B} as,

\begin{align}
B_{\normalfont\textbf{k},l}= \sum_{{{\normalfont\textbf{n}}}\not=\normalfont{\textbf{k}}}{\beta}_{\normalfont \textbf{n},l}\ t_0^{-|\normalfont\textbf{n}|} \ \normalfont\textbf{n!}\bigg(\prod_{\gamma\in \mathbb{Z}_{\geq 1}}  \mu^{\normalfont\textbf{n}}(\gamma)!\bigg)f_{\textbf{n}}(C_\textbf{k}),    
\end{align}
where the dependence of $B_{\textbf{k},l}$ on the chosen $C_{\textbf{k}}$ is carried entirely by the $f_{\textbf{n}}(C_\textbf{k})$. While our result proves only that the value of $B_{\textbf{k},l}$ is independent of the poset $C_{\textbf{k}}$, we observed in our computations that the value of $f_{\textbf{n}}(C_\textbf{k})$ is also independent of $C_\textbf{k}$, \textit{i.e.} $f_{\textbf{n}}(C_\textbf{k})=f_{\textbf{n}}(C_\textbf{k}')$ for any pair of posets $C_{\textbf{k}}$ and $C_{\textbf{k}}'$.

  \begin{table}[ht]
\begin{tabular}{ | m{2cm} | m{4cm}|m{3cm}|m{5cm}|} 
  \hline
$(\textbf{k},l)$ & $\frac{\sum_{Sh({\textbf{k}},l)} \prod_{i=1}^d (t_0+v^i_1t_1)}{\prod_{\gamma\in\mathbb{Z}_{\geq 1}}\mu^{{\textbf{k}}}(\gamma)!}$& $\frac{t_0^dB_{\textbf{k},l}}{\textbf{k!}\prod_{\gamma\in\mathbb{Z}_{\geq 1}}\mu^{{\textbf{k}}}(\gamma)!}$
& $\beta_{\textbf{k},l}$\\

\hline
 $(1,1,1)$& $3t_0^2+3t_0t_1+t_1^2$ & $3t_0^2+t_1t_0$& $\beta_{1,1,1}=t_1^2+2t_1t_0$\\
  \hline
 $(2,1,1)$& $12t_0^2+10t_0t_1+3t_1^2$ & $12t_0^2+3t_1t_0$& $\beta_{2,1,1}=7t_1t_0+3t_1^2$\\
  \hline
  $(1,1,2)$& $6t_0^2+7t_1+12t_0t_1$ & $6t_0^2+4t_0t_1$& $\beta_{1,1,2}=7t_1^2+8t_0t_1$\\
  \hline
  $(1,1,1,1)$& $4t_0^3+6t_0^2t_1+4t_0t_1^2+t_1^3$ & $8t_1t_0^2+4t_0^3+3t_1^2t_0$& $\beta_{1,1,1,1}=t_1^3-2t_0^2t_1+t_0t_1^2$\\
  \hline
  \end{tabular}
  \caption{Computing the coproduct coefficients in the un-normalised forest model (theorem \ref{thm for general set up}, clause \ref{clause forest weight}) using equations \eqref{unnormalised forest beta} and \eqref{unnormalised forest B}. The fourth column is obtained by subtracting the third column from the second. \label{table more betas}}
  \end{table}

\subsection{CSG models which are not sub-Hopf}
We now prove in a series of lemmas that  CSG models with one or both of $t_0$ and $t_1$ being greater than zero which are not contained in any of the families of proposition \ref{thm main} are not Hopf.

In each case below, let $a_n$ denote the generator of degree $n$, with $a_1=\bullet$. Let $c_n$ denote the corolla of degree $n$, \textit{i.e.} the poset which contains a single minimal element with $n-1$ elements directly above it. Let $\bar{c}_n$ denote the anti-corolla of degree $n$, \textit{i.e.} the poset which contains a single maximal element with $n-1$ elements directly below it. Let ${d_n}$ denote the poset we get by adding a unique maximal element to $c_{n-1}$, or equivalently adding a unique minimal element to $\bar{c}_{n-1}$ (\textit{e.g.} $d_4$ is the diamond). Let $l_n$ denote the $n$-element ladder or chain, \textit{i.e.} the poset in which all $n$ elements are related.

\begin{lemma}\label{originary not hopf k>1}
    An originary CSG model (\textit{i.e.} a model with $t_0=0,t_1\not=0$) is not Hopf if $t_k\not=0$ for some $k>1$.  
\end{lemma}
\begin{proof}
     For contradiction, consider an originary model with at least two non-vanishing couplings, and let $k$ be the smallest integer greater than 1 for which $t_{k}\not=0$. Now, $t_{k}\not=0$ implies $\mathbb{P}(d_{k+1})\not=0$ which implies $\Gamma(\bar{c}_{k-1},a_1)\not=0$. Since $\bar{c}_{k-1}$ is connected and contains more than one minimal element it can never be generated in this model and so the model is not Hopf.
\end{proof}

\begin{lemma}\label{lemma csg not hopf 1}
   Consider a CSG model with $t_0\not=0$. If there exists some $k\geq 1$ for which $t_k=0,t_{k+1}\not=0$ then the model is not Hopf.
\end{lemma}

\begin{proof}
        $t_{0}\not=0$ and $t_{k+1}\not=0$ implies  $\mathbb{P}(\bar{c}_{k+2})\not= 0$ which implies $\Gamma(\bar{c}_{k+1}, a_1)\not=0$. Note also that $t_k=0$ implies $\mathbb{P}(\bar{c}_{k+1})=0$. Then, since $\bar{c}_{k+1}$ is connected the model is not Hopf.
\end{proof}

\begin{lemma}\label{lemma csg not hopf 2}
Consider a CSG model with $t_0\not=0$. If $t_1\not=0$ and $t_2\not=0$, then the model is Hopf only if it is a Transitive Percolation model.

\end{lemma}

\begin{proof}
   Without loss of generality, similarly to the argument in the proof of Lemma~\ref{lem forest hopf by iso}, let $t_0=1$ and $t_1=t$.
   
       By direct computation we find,
        \begin{equation}\label{eq constraints on ts for Hopf}
            \begin{split}
                &\frac{\Gamma(l_3, a_1)}{\mathbb{P}(l_3)}=\frac{1}{\lambda(3,0)}\bigg(4+3\frac{t_2+t_3}{t+t_2}+2\frac{t_2(t+2t_2+t_3)}{t(t+t_2)}+t+2t_2+t_3\bigg) ,\\
                &\frac{\Gamma(c_3, a_1)}{\mathbb{P}({c}_3)}=\frac{1}{\lambda(3,0)}\bigg(4+t+5\frac{t_2}{t}+2t_2+2\frac{t_2^2}{t^2}+\frac{t_2^2}{t}\bigg) ,\\
                &\frac{\Gamma(\bar{c}_3, a_1)}{\mathbb{P}(\bar{c}_3)}=\frac{1}{\lambda(3,0)}\bigg(4+3(t+\frac{t_3}{t_2})+t^2+3\frac{t_3t}{t_2}+\frac{t_3t^2}{t_2}\bigg).
            \end{split}
        \end{equation}
        That the model be Hopf requires that the three lines above are equal to each other. Solving the equalities we find the unique solution $t_2=t^2, t_3=t^3$ for $t_2, t_3>0$.
        
        Suppose $t_k=t^k$ for all $k=0,1,2,\cdots,m$. 
        Then, by equating the expressions in \eqref{notTP notHopf}, we find that the model is Hopf only if $t_{m+1}=t^{m+1}$. To obtain the first (second) expression of \eqref{notTP notHopf}, we compute $\Gamma(c_{m+1}, a_1)$ ($\Gamma(\bar{c}_{m+1}, a_1)$) by summing over all posets which can be cut to give an upset $c_{m+1}$ ($\bar{c}_{m+1}$) and a downset $a_1$. Each poset contributes a factor equal to its probability times the number of allowed cuts. The allowed posets and their associated factors are shown in tables \ref{table V contributions} and \ref{table lambda contributions}.
        \begin{equation}\label{notTP notHopf}
            \begin{split}
                &\frac{\Gamma(c_{m+1}, a_1)}{\mathbb{P}({c}_{m+1})}=\frac{1}{\lambda(m+1,0)}\bigg[t^{m+1}(t+1)^m+2t^{2m}+(m+2)t^m\\
                & \hspace{8cm} +\sum_{i=1}^{m-1}t^{2m-i}\bigg(2\binom{m}{i}+\sum_{j=2}^{i+1}\binom{m-j+1}{m-i}\bigg)\bigg],\\
                &\frac{\Gamma(\bar{c}_{m+1}, a_1)}{\mathbb{P}(\bar{c}_{m+1})}=\frac{1}{\lambda(m+1,0)}\bigg[t^m(t^m+t_{m+1})+(m+1)t_{m+1}+(m+2)t^m\\
                &\hspace{8cm} +\sum_{i=1}^{m-1}\bigg(t^i(t^m+t_{m+1})\sum_{j=0}^{m-i}\binom{m-j}{i}\bigg)\bigg].
            \end{split}
        \end{equation}
 \end{proof}

\begin{table}[ht]
\begin{tabular}{ | m{3cm} | m{3cm}| m{3cm} |  m{3cm} |} 
  \hline
  
poset & probability of each natural labeling & number of natural labelings & number of cuts \\ 
  \hline

\begin{tikzpicture}
\draw[gray, thick] (0,0) -- (0.5,-0.75);
\draw[gray, thick] (1,0) -- (0.5,-0.75);
\draw[gray, thick] (0.5,-0.75) -- (0.5,-1.5);
\filldraw[black] (0,0) circle (2pt) node[anchor=west]{};
\filldraw[black] (1,0) circle (2pt) node[anchor=west]{};
\filldraw[black] (0.5,-1.5) circle (2pt) node[anchor=west]{};
\filldraw[black] (0.5,-0.75) circle (2pt) node[anchor=west]{};
\filldraw[black] (0.5,0) circle (0.6pt) node[anchor=west]{};
\filldraw[black] (0.75,0) circle (0.6pt) node[anchor=west]{};
\filldraw[black] (0.25,0) circle (0.6pt) node[anchor=west]{};
\draw [decorate,
    decoration = {brace}] (0,.2) --  (1,0.2)
    node[pos=0.5,above=5pt,black]{$m$};
\end{tikzpicture} & $\frac{{t_1(t_1+t_2)^m}}{\prod_{x=1}^{m+1}\lambda(x,0)}$ & 1 &1\\ 
  \hline
  
\begin{tikzpicture}
\draw[gray, thick] (0,0) -- (0,-0.75);
\draw[gray, thick] (1,0) -- (0,-0.75);
\draw[gray, thick] (0,0) -- (1,-0.75);
\draw[gray, thick] (1,0) -- (1,-0.75);

\filldraw[black] (0,0) circle (2pt) node[anchor=west]{};
\filldraw[black] (1,0) circle (2pt) node[anchor=west]{};

\filldraw[black] (0,-0.75) circle (2pt) node[anchor=west]{};
\filldraw[black] (1,-0.75) circle (2pt) node[anchor=west]{};

\filldraw[black] (0.5,0) circle (0.6pt) node[anchor=west]{};
\filldraw[black] (0.75,0) circle (0.6pt) node[anchor=west]{};
\filldraw[black] (0.25,0) circle (0.6pt) node[anchor=west]{};
\draw [decorate,
    decoration = {brace}] (0,.2) --  (1,0.2)
    node[pos=0.5,above=5pt,black]{$m$};
\end{tikzpicture}& $\frac{t_0t_2^m}{\prod_{x=1}^{m+1}\lambda(x,0)}$ & 1 & 2\\ 
 
  \hline
\begin{tikzpicture}
\draw[gray, thick] (0.7,0) -- (0,-0.75);
\draw[gray, thick] (1.5,0) -- (0,-0.75);
\draw[gray, thick] (0.7,0) -- (1,-0.75);
\draw[gray, thick] (1.5,0) -- (1,-0.75);

\filldraw[black] (0.7,0) circle (2pt) node[anchor=west]{};
\filldraw[black] (1.5,0) circle (2pt) node[anchor=west]{};

\filldraw[black] (0,-0.75) circle (2pt) node[anchor=west]{};
\filldraw[black] (1,-0.75) circle (2pt) node[anchor=west]{};

\filldraw[black] (1,0) circle (0.6pt) node[anchor=west]{};
\filldraw[black] (1.25,0) circle (0.6pt) node[anchor=west]{};
\filldraw[black] (0.75,0) circle (0.6pt) node[anchor=west]{};
\draw [decorate,
    decoration = {brace}] (0.6,.2) --  (1.5,0.2)
    node[pos=0.5,above=5pt,black]{$l$};

\draw[gray, thick] (-0.5,0) -- (0.,-0.75);
\draw[gray, thick] (0.3,0) -- (0.,-0.75);

\filldraw[black] (-0.5,0) circle (2pt) node[anchor=west]{};
\filldraw[black] (0.3,0) circle (2pt) node[anchor=west]{};

\filldraw[black] (0,0) circle (0.6pt) node[anchor=west]{};
\filldraw[black] (0.25,0) circle (0.6pt) node[anchor=west]{};
\filldraw[black] (-0.25,0) circle (0.6pt) node[anchor=west]{};
\draw [decorate,
    decoration = {brace}] (-0.5,.2) --  (0.4,0.2)
    node[pos=0.5,above=5pt,black]{$k$}; 
    
\end{tikzpicture} with $k+l=m$ \newline and $1\leq k< m$ & $\frac{t_0t_{1}^kt_2^l}{\prod_{x=1}^{m+1}\lambda(x,0)}$ & $2\binom{m}{k}$\newline$+\sum_{j=2}^{k+1}\binom{m-j+1}{l}$ & 1 \\ 
  \hline
\begin{tikzpicture}
\draw[gray, thick] (0,0) -- (0.5,-0.75);
\draw[gray, thick] (1,0) -- (0.5,-0.75);

\filldraw[black] (0,0) circle (2pt) node[anchor=west]{};
\filldraw[black] (1,0) circle (2pt) node[anchor=west]{};

\filldraw[black] (0.5,-0.75) circle (2pt) node[anchor=west]{};

\filldraw[black] (1.5,-0.75) circle (2pt) node[anchor=west]{};
\filldraw[black] (0.5,0) circle (0.6pt) node[anchor=west]{};
\filldraw[black] (0.75,0) circle (0.6pt) node[anchor=west]{};
\filldraw[black] (0.25,0) circle (0.6pt) node[anchor=west]{};
\draw [decorate,
    decoration = {brace}] (0,.2) --  (1,0.2)
    node[pos=0.5,above=5pt,black]{$m$};
\end{tikzpicture}
 & $\frac{t_0t_1^m}{{\prod_{x=1}^{m+1}\lambda(x,0)}}$ & $m+2$ &1\\ 
  \hline
\end{tabular}
 \caption{The contributions to $\Gamma(c_{m+1},a_1)$. Each poset contributes a factor equal to the product of the entries in its row with $t_0=1$ and $t_k=t^k$ for $k=1,\cdots, m$. \label{table V contributions}}
\end{table}

\begin{table}
\begin{tabular}{ | m{3.5cm} | m{3cm}| m{3cm} |  m{3cm} |} 
  \hline
  
poset & probability of each natural labeling & number of natural labelings & number of cuts \\ 
  \hline

\begin{tikzpicture}
\draw[gray, thick] (0,0) -- (0.5,-0.75);
\draw[gray, thick] (1,0) -- (0.5,-0.75);
\draw[gray, thick] (0,0) -- (0.5,0.75);
\draw[gray, thick] (1,0) -- (0.5,0.75);

\filldraw[black] (0,0) circle (2pt) node[anchor=west]{};
\filldraw[black] (1,0) circle (2pt) node[anchor=west]{};

\filldraw[black] (0.5,-0.75) circle (2pt) node[anchor=west]{};

\filldraw[black] (0.5,0.75) circle (2pt) node[anchor=west]{};
\filldraw[black] (0.5,0) circle (0.6pt) node[anchor=west]{};
\filldraw[black] (0.75,0) circle (0.6pt) node[anchor=west]{};
\filldraw[black] (0.25,0) circle (0.6pt) node[anchor=west]{};
\draw [decorate,
    decoration = {brace}] (-.1,0.1) --  (1.1,0.1)
    node[pos=0.5,above=1pt,black]{$m$};
\end{tikzpicture} & $\frac{t_0 t_1^m(t_m+t_{m+1})}{\prod_{x=0}^{m+1}\lambda(x,0)}$ & 1 &1\\ 
  \hline
  
\begin{tikzpicture}
\draw[gray, thick] (0,0) -- (0.5,-0.75);
\draw[gray, thick] (1,0) -- (0.5,-0.75);
\draw[gray, thick] (0,0) -- (0.5,0.75);
\draw[gray, thick] (1,0) -- (0.5,0.75);

\draw[gray, thick] (1.3,0) -- (0.5,0.75);
\draw[gray, thick] (2.3,0) -- (0.5,0.75);

\filldraw[black] (0,0) circle (2pt) node[anchor=west]{};
\filldraw[black] (1,0) circle (2pt) node[anchor=west]{};

\filldraw[black] (1.3,0) circle (2pt) node[anchor=west]{};
\filldraw[black] (2.3,0) circle (2pt) node[anchor=west]{};

\filldraw[black] (0.5,-0.75) circle (2pt) node[anchor=west]{};

\filldraw[black] (0.5,0.75) circle (2pt) node[anchor=west]{};
\filldraw[black] (0.5,0) circle (0.6pt) node[anchor=west]{};
\filldraw[black] (0.75,0) circle (0.6pt) node[anchor=west]{};
\filldraw[black] (0.25,0) circle (0.6pt) node[anchor=west]{};
\draw [decorate,
    decoration = {brace}] (-.1,0.1) --  (1.1,0.1)
    node[pos=0.5,above=1pt,black]{$k$};
    
\filldraw[black] (1.8,0) circle (0.6pt) node[anchor=west]{};
\filldraw[black] (2.05,0) circle (0.6pt) node[anchor=west]{};
\filldraw[black] (1.55,0) circle (0.6pt) node[anchor=west]{};
\draw [decorate,
    decoration = {brace,mirror}] (1.2,-0.1) --  (2.4,-0.1)
    node[pos=0.5,below=1pt,black]{$l$};
\end{tikzpicture} with $k+l=m$\newline and $1\leq k< m$& $\frac{t_0^{l+1}t_1^k(t_m+t_{m+1})}{\prod_{x=0}^{m+1}\lambda(x,0)}$ & $\sum_{j=0}^{l}\binom{m-j}{k}$ & 1\\ 
 
  \hline
 \begin{tikzpicture}
\draw[gray, thick] (0,0) -- (0.5,0.75);
\draw[gray, thick] (1,0) -- (0.5,0.75);
\draw[gray, thick] (1.5,0) -- (0.5,0.75);
\filldraw[black] (0,0) circle (2pt) node[anchor=west]{};
\filldraw[black] (1,0) circle (2pt) node[anchor=west]{};
\filldraw[black] (1.5,0) circle (2pt) node[anchor=west]{};
\filldraw[black] (0.5,0.75) circle (2pt) node[anchor=west]{};
\filldraw[black] (0.5,0) circle (0.6pt) node[anchor=west]{};
\filldraw[black] (0.75,0) circle (0.6pt) node[anchor=west]{};
\filldraw[black] (0.25,0) circle (0.6pt) node[anchor=west]{};
\draw [decorate,
    decoration = {brace,mirror}] (-.1,-0.1) --  (1.1,-0.1)
    node[pos=0.5,below=1pt,black]{$m$};
\end{tikzpicture}& $\frac{t_0^{m+1}t_{m+1}}{\prod_{x=0}^{m+1}\lambda(x,0)}$ & 1 & $m+1$ \\ 
  \hline

 \begin{tikzpicture}
\draw[gray, thick] (0,0) -- (0.5,0.75);
\draw[gray, thick] (1,0) -- (0.5,0.75);
\filldraw[black] (0,0) circle (2pt) node[anchor=west]{};
\filldraw[black] (1,0) circle (2pt) node[anchor=west]{};
\filldraw[black] (1.5,0) circle (2pt) node[anchor=west]{};
\filldraw[black] (0.5,0.75) circle (2pt) node[anchor=west]{};
\filldraw[black] (0.5,0) circle (0.6pt) node[anchor=west]{};
\filldraw[black] (0.75,0) circle (0.6pt) node[anchor=west]{};
\filldraw[black] (0.25,0) circle (0.6pt) node[anchor=west]{};
\draw [decorate,
    decoration = {brace,mirror}] (-.1,-0.1) --  (1.1,-0.1)
    node[pos=0.5,below=1pt,black]{$m$};
\end{tikzpicture} & $\frac{t_0^{m+1}t_m}{\prod_{x=0}^{m+1}\lambda(x,0)}$ & $m+2$ &1\\ 
  \hline

\end{tabular}
  \caption{The contributions to $\Gamma(\bar{c}_{m+1},a_1)$. Each poset contributes a factor equal to the product of the entries in its row with $t_0=1$ and $t_k=t^k$ for $k=1,\cdots, m$.}
\label{table lambda contributions}
\end{table}

\FloatBarrier

\subsection{Proofs of the main propositions and theorem}\label{subsec thm proof}

\begin{proof}[Proof of proposition \ref{thm main}]
    By lemma \ref{tp lemma}, the Transitive Percolation models are Hopf. By lemma \ref{lem forest hopf by iso}, the Forest models are Hopf. The Tree model generators $a_n$ are related to the Connes-Moscovici Hopf algebra generators $\delta_n$ via $\delta_n=\Gamma(n)a_n$. Therefore, the Tree model is Hopf. The Dust model is the model in which each generator $a_n$ is the $n$-antichain, which is known to be Hopf.
    
    By lemma \ref{originary not hopf k>1}, an originary CSG model ($t_0=0)$ is only Hopf if it is the Tree model. By lemma \ref{lemma csg not hopf 1}, a CSG model with $t_0>0$ and $t_2=0$ cannot be Hopf unless it is a Forest model or the Dust model. By lemma \ref{lemma csg not hopf 1}, a CSG model with $t_0>0$ and $t_2>0$ cannot be Hopf unless $t_1>0$. By lemma \ref{lemma csg not hopf 2}, a CSG model with $t_0>0,t_1>0$ and $t_2>0$ cannot be Hopf unless it is a Transitive Percolation model.
\end{proof}

\begin{proof}[Proof of proposition \ref{prop converse}]
    Given a poset $C_{N+1}$ of cardinality $N+1$, note that $(i)$ $\mathbb{P}(C_{N+1})\not=0$ only if $C_{N+1}$ contains a unique maximal element. Note also that $(ii)$ there exists at least one connected poset $C_{N+2}$ with two maximal elements such that $\mathbb{P}(C_{N+2})\not=0$. Then $(ii)$ implies that there exists some connected poset $C_{N+1}$ with 2 maximal elements for which $\Gamma(C_{N+1},\bullet)\not=0$. By $(i)$, such poset does not appear in $a_{N+1}$.
\end{proof}

\begin{proof}[Proof of theorem \ref{thm for general set up}]
    By extension of the reasoning in Lemma~\ref{lemma M}, note that when $s_0=~s\not=~0$, $s_i=0$ for all $i>0$ and $t_i=t\not=0$ for all $i$ we have $ M_{\substack{s_0, s_1, \ldots\\t_0, t_1, \ldots}}=stB_+
  $. Therefore, $A(x)= x st B_+(f(A(x)))$ which is sub-Hopf if and only if $X(x)=xB_+(f(A(x)))$ is sub-Hopf. The conditions for $X(x)$ to be sub-Hopf are given in proposition \ref{prop Foissy_characterization_DSE}.

Now, when $s_i=s\not=0$ for all $i$ and $t_i\geq 0$ for all $i$, \[M_{\substack{s, s, s, \ldots\\t_0, t_1, \ldots}}(C_n)=s\sum_{C_{n+1}}w(C_{n}\rightarrow C_{n+1})C_{n+1}.
        \]
Suppose $a_n=s^n\sum_{C_n}w(C_n)C_n$, where the sum is over all posets with $n$ elements and $w$ is the weight as defined in \eqref{weight}. Then using the properties of $w$ one can show that when $f(u)=1+u$, \begin{equation*}\label{last eqtn}
a_{n+1}=s^n\sum_{C_n}w(C_n) M_{\substack{s, s, s, \ldots\\t_0, t_1, \ldots}}(C_n)=s^{n+1}\sum_{C_{n+1}}w(C_{n+1})C_{n+1}.\end{equation*} 

When at least one of $t_0$ and $t_1$ is greater than zero we find $a_n=s^n\prod_{x=1}^{n-1}\lambda(x,0)\sum_{C_n}\mathbb{P}(C_n) C_n$ for all $n>0$. Since a rescaling of the generators has no effect on the algebra, given set of couplings $t_0,t_1,\ldots$ the algebra generated by these $a_n$ is Hopf if and only if the generators given by $\sum_{C_n}\mathbb{P}(C_n) C_n$ with the same couplings are Hopf. The conditions for such models to be Hopf are given in proposition \ref{thm main}.

Now suppose $t_0=t_1=0$ and let $t_N$ denote the first non-zero coupling. By similar argument to the above, when $f(u)=1+u$, for all $n>N$, $a_n$ is proportional to the generators of proposition \ref{prop converse}, and the result follows.
\end{proof}

\section{Conclusions}
The CSG dynamics are the archetype of spacetime dynamics in causal set theory. In this work, we were able to set the CSG dynamics into a wider context and explore their algebraic properties. Searching for new formulations of spacetime dynamics (especially fully quantal dynamics) has been a long-time research goal of causal set theory (see for example \cite{Dowker:2019qiz, Surya:2020cfm, Wuthrich:2015vva}), and since equation \eqref{eq general set up} plays the role of an equation of motion of sorts for the causal set spacetime it may hold new opportunities for doing so. In particular, requiring that a dynamics is Hopf in the sense of Theorem~\ref{thm for general set up} could become a new physically-motivated requirement on our dynamics.

The reason that it is physically interesting for solutions of Dyson-Schwinger equations to give subHopf algebras is that the Connes-Kreimer Hopf algebra is used to mathematically encode renormalization, so a solution giving a subHopf algebra means that the solution to the Dyson-Schwinger equation can be renormalized without needing any diagrams or combinations of diagrams that aren't built by the Dyson-Schwinger equation itself.  This is as it should be physically, and so it is gratifying that the cases that show up in physics are among those in Foissy's characterization Proposition~\ref{prop Foissy_characterization_DSE}. To put it another way, the result is telling us that the Dyson-Schwinger equations which appear in physics are compatible with renormalization.  One further aspect of this compatibility is that if $X(x)$ is the solution to the Dyson-Schwinger equation, then there are formulas for the coproduct $\Delta(X(x))$ as a whole, without needing to work term by term.

Thus, to try to understand the second part of Theorem~\ref{thm for general set up} physically, we need to understand what the poset Hopf algebra's physical interpretation in causal set theory should be and then we can understand Theorem~\ref{thm for general set up} as telling us that the particular models it picks out have a compatibility with this interpretation of the Hopf algebra. The Hopf algebra works by cutting a causal set into a past and a future (not in the narrower sense of a past or future of a single element, but the a broader sense of simply a down-set and an up-set) in all possible ways. Indeed, down-sets (or ``stems'') are known to play an important role in causal set theory where they play the role of observables, see for example \cite{Brightwell:2002vw, Dowker:2022ehl, Dowker:2019qiz}. If we could formulate this or other physical statements algebraically in terms of the poset Hopf algebra, for instance using M\"obius inversion, then these interpretations in causal set theory would continue to make sense directly on the growth models we found to be Hopf. In other words, Theorem~\ref{thm for general set up} is telling us that transitive percolation, along with the tree, forest and dust models, are distinguished by the fact that properties expressible in terms of the Hopf algebra all make sense directly on these models as well. 

On the mathematical side of things, one direction for future investigation stemming from this work is finding the full characterization of when solutions to \eqref{eq general set up} give subHopf algebras.  To this effect, preliminary investigations which we carried out suggest which particular generalizations of Theorem~\ref{thm for general set up} might be most accessible.

As a first step, one can consider the models with $s_0\not=0, s_i=0$ for all $i>0$, while leaving $t_0,t_1,\ldots$ unconstrained. When $f(u)=1+u$ this gives $a_n=l_n s^{n}\prod_{i=0}^{n-1}t_i$, where $l_n$ is the ladder (or chain) with $n$ elements, so the model is Hopf for any $t_0,t_1,\ldots$. Setting $f=\frac{1}{1-\bullet}$ we get $a_n=c_n s^{n}t_{n-1}$, where $c_n$ is a corolla, so the model is Hopf as long as $t_i\not=0$ for all $i$. We conjecture that for a generic $f(u)$, each poset $C_n$ is weighted in $a_n$ by the product of three numbers: the CSG probability of a \emph{labelled representative of it}, $\mathbb{P}(\lc{C}_n)$, the coefficient with which $C_n$ appears in $X(x)=B_+(f(X(x)))$, and $s^n$. 

One can also look at relaxing the constraint that the $t_i$ are positive reals.  In this case there are three additional solutions to \eqref{eq constraints on ts for Hopf}.  Further constraints will be introduced by the equations at the next order and many of these solutions will not end up giving Hopf models, though some cases such as $t_i = (-1)^{i}$ are Hopf.

 Additionally, our proofs for the CSG models which do not give Hopf algebras do not make deep use of the shape of $f(u)$ and so suggests the possibility that with quite mild hypotheses (excluding for instance when $f(u)$ is a constant) being Hopf for nonlinear $f(u)$ may require being Hopf for $f(u)=1+u$, which would be very helpful towards a general characterization. Considering varying $f(u)$ within the CSG models which are Hopf would also be interesting.
  
Other special cases, including $f(u)=1+u$ while leaving the $s_i$ and the $t_i$ both unconstrained, can be explored computationally, recursively generating the first half a dozen or so terms of the solution to \eqref{eq general set up} with particular parameters and then computing the coproducts of these terms. 

Other combinatorial puzzles arising from this work are, finding a closed-form algebraic expression for the coproduct coefficients shown in Table~\ref{table beta forests} (via a method akin to that we used in the proof to Lemma~\ref{tp lemma}), and proving that $f_n(C_{\textbf{k}})$ as defined in \eqref{fn eqtn} depends only on $\textbf{k}$.  

\bibliographystyle{unsrt}
\bibliography{poset_hopf,causet_refernces}

\end{document}